\documentclass[11pt]{amsart}

\usepackage[T1]{fontenc}
\usepackage[utf8]{inputenc}
\usepackage{lmodern}

\usepackage[american]{babel}
\usepackage[babel, final]{microtype}
\usepackage{amssymb}
\usepackage{mathdots}

\usepackage{xifthen}
\usepackage{mathtools}

\usepackage[pdftex, dvipsnames]{xcolor}
\usepackage[pdftex, final]{graphicx}
\usepackage{caption}
\usepackage{subcaption}
\usepackage{pinlabel}
\usepackage[pdftex,%
  letterpaper,%
  includehead,%
  includefoot,%
  nomarginpar,%
  lmargin=1in,%
  rmargin=1in,%
  tmargin=1in,%
  bmargin=1in,%
]{geometry}
\usepackage[pdftex,%
  final,%
  colorlinks=true,%
  linkcolor=NavyBlue,%
  citecolor=NavyBlue,%
  filecolor=NavyBlue,%
  menucolor=NavyBlue,%
  urlcolor=NavyBlue,%
  bookmarks=true,%
  bookmarksdepth=3,%
  bookmarksnumbered=true,%
  bookmarksopen=true,%
  bookmarksopenlevel=2,%
]{hyperref}
\hypersetup{
  pdftitle={Profinite tensor powers},
  pdfauthor={David Treumann and C.-M. Michael Wong},
  pdfsubject={Profinite tensor powers},
  pdfkeywords={Tensor product, profinite group, profinite cover, Heegaard Floer 
    homology}
}

\usepackage[all]{xy}
\usepackage{tikz}

\makeatletter
\g@addto@macro\@floatboxreset\centering
\makeatother

\allowdisplaybreaks[4]

\let\originalleft\left
\let\originalright\right
\renewcommand{\left}{\mathopen{}\mathclose\bgroup\originalleft}
\renewcommand{\right}{\aftergroup\egroup\originalright}

\theoremstyle{definition}
\newtheorem*{theorem*}{Theorem}
\newtheorem*{proposition*}{Proposition}
\newtheorem*{remark*}{Remark}

\numberwithin{equation}{subsection}

\newcommand*{\set}[1]{\left\{#1\right\}}
\newcommand*{\gen}[1]{\left\langle #1 \right\rangle}

\renewcommand*{\setminus}{\mathbin{-}}

\DeclareMathOperator{\card}{card}

\newcommand*{\comp}{\circ}

\newcommand*{\bigdirsum}{\bigoplus}
\newcommand*{\isom}{\cong}
\newcommand*{\tensor}{\otimes}
\newcommand*{\bigtensor}{\bigotimes}

\newcommand*{\cc}{\mathrm{cc}}
\newcommand*{\mcc}{\mathrm{mcc}}
\newcommand*{\bigprofintensor}[1]{\bigtensor_{#1}^\mcc}

\newcommand*{\blank}{\mathord{-}}

\newcommand*{\id}{\mathrm{id}}

\newcommand*{\specialset}[1]{\mathbf{#1}}
\newcommand*{\F}{\specialset{F}}
\newcommand*{\Z}{\specialset{Z}}
\newcommand*{\Q}{\specialset{Q}}
\newcommand*{\R}{\specialset{R}}
\newcommand*{\C}{\specialset{C}}
\newcommand*{\Ftwo}{\F_2}

\newcommand*{\vsV}{V}
\newcommand*{\vsW}{W}
\newcommand*{\bimodV}{V}

\newcommand*{\grpG}{G}
\newcommand*{\grpH}{H}
\newcommand*{\grpK}{K}

\newcommand*{\colgrp}{\grpH}
\newcommand*{\rowgrp}{\grpK}

\newcommand*{\subgrp}{\leq}
\newcommand*{\nsubgrp}{\trianglelefteq}
\newcommand*{\osubgrp}{\leq_{\mathrm{o}}}
\newcommand*{\onsubgrp}{\trianglelefteq_{\mathrm{o}}}
\newcommand*{\onsupgrp}{\trianglerighteq_{\mathrm{o}}}

\newcommand*{\catprofgrp}{\mathcal{P}}
\newcommand*{\groups}{\mathcal{S}}
\newcommand*{\germ}{\mathfrak{g}}

\newcommand*{\basis}[1]{#1}
\newcommand*{\basisB}{\basis{B}}
\newcommand*{\basisC}{\basis{C}}
\newcommand*{\basisD}{\basis{D}}
\newcommand*{\graphbasis}[1]{#1}
\newcommand*{\graphB}{\graphbasis{B}}
\newcommand*{\graphC}{\graphbasis{C}}
\newcommand*{\graphD}{\graphbasis{D}}

\newcommand*{\sector}[1]{\mathcal{#1}}
\newcommand*{\sectorA}{\sector{A}}

\newcommand*{\lmod}[2]{#2 \backslash #1}

\DeclareMathOperator{\Spec}{Spec}

\DeclareMathOperator{\Fun}{Fun}
\DeclareMathOperator{\Mod}{Mod}
\DeclareMathOperator{\Mat}{Mat}

\DeclareMathOperator{\Sym}{Sym}
\DeclareMathOperator{\Homeo}{Homeo}

\newcommand*{\bdy}{\partial}

\newcommand*{\zero}{0}
\newcommand*{\one}{1}

\newcommand*{\knotL}{L}
\newcommand*{\mfd}{M}
\newcommand*{\mer}{\mu}
\newcommand*{\longt}{\lambda}

\newcommand*{\qisom}{\simeq}

\newcommand*{\surf}{F}
\newcommand*{\basept}{z}
\newcommand*{\sbdiff}{\phi}
\newcommand*{\sbdiffmfd}[1][]{%
  \ifthenelse{\isempty{#1}}%
  {Y (\sbdiff)}%
  {Y (\sbdiff^{#1})}%
}

\newcommand*{\figeight}{4_1}

\newcommand*{\twistl}{\tau_{a}}
\newcommand*{\twistm}{\tau_{b}}

\DeclareMathOperator{\Spin}{Spin}
\newcommand*{\Spinc}{\Spin^c}
\newcommand*{\Spincstr}[1]{\mathfrak{#1}}
\newcommand*{\Spinct}{\Spincstr{t}}

\DeclareMathOperator{\CF}{CF}
\DeclareMathOperator{\SFC}{SFC}
\DeclareMathOperator{\SFH}{SFH}
\DeclareMathOperator{\HFK}{HFK}
\DeclareMathOperator{\CFDA}{CFDA}
\newcommand*{\HFKh}{\widehat{\HFK}}
\newcommand*{\CFDAh}{\widehat{\CFDA}}
\newcommand*{\CFDAhsbdiff}[1][]{%
  \ifthenelse{\isempty{#1}}%
  {\CFDAh (\sbdiff)}%
  {\CFDAh (\sbdiff^{#1})}%
}

\newcommand*{\Ainf}{A_{\infty}}
\newcommand*{\DA}{DA}
\renewcommand*{\AA}{AA}
\newcommand*{\DD}{DD}
\newcommand*{\AD}{AD}

\newcommand*{\heeg}{\mathcal{H}}
\newcommand*{\pmc}{\mathcal{Z}}

\newcommand*{\dga}{A}
\newcommand*{\idemring}{I}
\newcommand*{\idem}{\iota}
\newcommand*{\strand}{\rho}

\newcommand*{\boxtensor}{\boxtimes}

\newcommand*{\generator}[1]{\mathbf{#1}}
\newcommand*{\genp}{\generator{p}}
\newcommand*{\genq}{\generator{q}}
\newcommand*{\genr}{\generator{r}}

\newcommand*{\genf}{\generator{f}}
\newcommand*{\geng}{\generator{g}}
\newcommand*{\genh}{\generator{h}}

\newcommand*{\genpmi}{\genp}
\newcommand*{\genpl}{\genf}
\newcommand*{\genqmi}{\genq}
\newcommand*{\genql}{\geng}
\newcommand*{\genrmi}{\genr}
\newcommand*{\gensl}{\genh}

\newcommand*{\genpp}{\genpmi\genpl}
\newcommand*{\genps}{\genpmi\gensl}
\newcommand*{\genqq}{\genqmi\genql}
\newcommand*{\genrs}{\genrmi\gensl}
\newcommand*{\genrp}{\genrmi\genpl}

\newcommand*{\gent}{\generator{t}}
\newcommand*{\genu}{\generator{u}}
\newcommand*{\genv}{\generator{v}}
\newcommand*{\genw}{\generator{w}}
\newcommand*{\genx}{\generator{x}}
\newcommand*{\geny}{\generator{y}}
\newcommand*{\genz}{\generator{z}}

\newcommand*{\DAmap}{\delta}

\newcommand*{\bimodmerinv}{B}
\newcommand*{\bimodlong}{A}

\newcommand*{\aug}{\epsilon}

\DeclareMathOperator{\HH}{HH}
\DeclareMathOperator{\CH}{CH}
\newcommand*{\HHbox}{\widetilde{\HH}}
\newcommand*{\CHbox}{\widetilde{\CH}}

\newcommand*{\outerloop}{\mathcal{L}}
\newcommand*{\outerloopp}{\outerloop_+}

\title{Profinite tensor powers}

\author{David Treumann}
\address{Department of Mathematics \\ Boston College \\ Chestnut Hill, MA 02467 
  \\ USA}
\email{\href{mailto:david.treumann@bc.edu}{david.treumann@bc.edu}}

\author{C.-M. Michael Wong}
\address{Department of Mathematics and Statistics \\ University of Ottawa \\ 
  Ottawa, ON K1N 6N5 \\ Canada}
\email{\href{mailto:Mike.Wong@uOttawa.ca}{Mike.Wong@uOttawa.ca}}

\begin{document}

\begin{abstract}
  We discuss the problem of defining a tensor product of
  profinitely many copies of a vector space $\vsV$, and propose a definition
  $\bigprofintensor{X} \vsV$ in the special situation that (1) $\vsV$ is
  finite-dimensional over $\F_2$, and (2) the profinite $X$ indexing the
  tensor factors is acted on with finitely many orbits by a pro-$2$-group.
  The ``mcc'' on the tensor sign stands for ``magnetized and conditionally
  convergent.'' A variant construction makes sense when $\vsV$ is a bimodule
  over a ring of the form $\F_2 \times \dotsb \times \F_2$, and the index set 
  $X$ has the profinite version of a cyclic order. The definition organizes 
  some computations in Heegaard Floer homology: it can be pitched as a 
  computation of the Heegaard Floer theory of some pro-$3$-manifolds, though we 
  do not know how to define such a thing.
\end{abstract}

\maketitle

\section{Introduction}
\label{sec:intro}

\subsection{Infinite tensor products}
\label{ssec:intro-vs-infin}
Infinite tensor products are routinely considered in quantum theory, but they 
do not have a readymade definition from abstract algebra. One wants 
$\bigtensor_{x \in X} \vsV_x$ to be spanned by pure tensors
\begin{equation}
  \label{eq:intro-pure-tensor}
  \bigtensor_{x \in X} v_x, \qquad v_x \in \vsV_x,
\end{equation}
but when $X$ is an infinite set, the multilinear relations that one would like 
to impose among the \eqref{eq:intro-pure-tensor} involve infinite sums 
(consider what happens if you make infinitely many substitutions $v_x = v'_x + 
v''_x$) and infinite products (if you make infinitely many substitutions $v_x = 
c_x v'_x$). Most approaches to making $\bigtensor_{x \in X} \vsV_x$ 
precise\footnote{Not the approach of \cite{Ng13}.} impose additional structure 
to make it possible to analyze the convergence of the sums and products.

One standard way to do it \cite[Section~4.1]{vNe39}, when $X$ is countable and 
every $\vsV_x$ is a Hilbert space, is to choose a unit vector $|0\rangle_x$ in 
each $\vsV_x$ and to study only those pure tensors \eqref{eq:intro-pure-tensor} 
that have all but finitely many of the $v_x$ equal to the ``base point'' 
$|0\rangle_x$:
\begin{equation}
  \label{eq:intro-pure-tensor-res}
  v_x = |0\rangle_x \text{ for all but finitely many $x \in X$}.
\end{equation}
Only finitely many substitutions $v_x = v'_x +v''_x$ or $v_x = c_x v'_x$ are 
possible without leaving this restricted class of pure tensors, and the 
relations implied by these substitutions make sense in pure algebra. The inner 
products on the $\vsV_x$ separately induce an inner product on the vector space 
obtained by applying these relations, and $\bigtensor_{x \in X} \vsV_x$ is 
often taken to be the Hilbert space completion of this vector space. The 
outcome is a separable Hilbert space if all of the $\vsV_x$ are separable 
Hilbert spaces.

\subsection{Profinite tensor products}
\label{ssec:intro-vs-profin}
In this paper we seek a new definition of $\bigtensor_{x \in X} \vsV_x$ when 
$X$ is a profinite set. It goes against the grain of profinite set theory to 
impose a condition that refers to ``all but finitely many'' of the elements of 
a profinite set, like \eqref{eq:intro-pure-tensor-res}.  The restriction on 
pure tensors that we will consider attends not to finite subsets of $X$ but to 
open subsets of $X$. For simplicity, suppose that all of the $\vsV_x$ are the 
same vector space, $\vsV_x = \vsV$, so that a pure tensor 
\eqref{eq:intro-pure-tensor} is the same data as a function $X \to \vsV$. The 
class of pure tensors that we are interested in are those that obey the 
following condition:
\begin{equation}
  \label{eq:intro-magnetic}
  x \mapsto v_x \text{ is locally constant.}
\end{equation} For example, if $\vsV$ has a two-element basis whose elements 
are labeled ``up'' and ``down,'' and there is a site $x \in X$ where $v_x$ is 
``down'', then $v_y$ is ``down'' also for every $y$ that is sufficiently close 
to $x$. Because of this similarity to ferromagnetic spin configurations, we 
call the class of pure tensors obeying \eqref{eq:intro-magnetic} ``magnetized 
pure tensors.''

Is there a reasonable definition of a tensor product that is topologically 
spanned by the magnetized pure tensors? In this paper we give a definition in a 
very special situation: when
\begin{itemize}
  \item $\vsV$ is a finite-dimensional vector space over $\F_2$, the field with 
    $2$ elements, and
  \item there is a profinite group $G$ of $2$-power order that acts 
    continuously on $X$ with finitely many orbits.
\end{itemize}
The role of the field $\F_2$ is that for trivial but inimitable reasons the 
product of any infinite collection $(c_x \in \F_2)_{x \in X}$ has a clear 
definition: it is $1$ if every $c_x$ is equal to $1$, and it is $0$ if any 
$c_x$ is not equal to $1$. Even so, the substitutions $v_x = v'_x + v''_x$ lead 
to infinite sums of magnetized pure tensors. The role of the group action is to 
specify which of these sums is ``conditionally convergent'', by matching pairs 
of cancelling summands. We denote our construction $\bigprofintensor{X}\vsV$, 
the ``mcc'' stands for ``magnetized and conditionally convergent.'' 

\subsection{Germs of groups and germs of actions}
\label{ssec:intro-germs}
We wish to view $\bigprofintensor{X}\vsV$ as a space of quantum fields on $X$.  
There are several weak points to this analogy, one of which is that the notion 
of a conditionally convergent tensor is sensitive to the action $G \times X \to 
X$, and the data of this action is not a local structure on $X$. We address 
this by noting that the action determines a structure on $X$, a ``germ 
action'', which is local, and that the class of conditionally convergent 
tensors is determined by the germ action.

Two profinite groups $G$ and $G'$ are said to have isomorphic germs if there is 
an open subgroup of $G$ that is isomorphic to an open subgroup of $G'$. A 
category of germs can be defined as a localization of the category of profinite 
groups: if $\catprofgrp$ is the category of profinite groups and continuous 
homomorphisms, and $W$ is the class of morphisms that are open inclusions, then 
the category of germs is $\catprofgrp[W^{-1}]$. Results of Lazard \cite{Laz65} 
show that $\catprofgrp[W^{-1}]$ has (for each prime $p$) a full subcategory 
that's equivalent to the category of finite-dimensional Lie algebras over 
$\Q_p$, and so we denote germs in the same Fraktur font used to denote Lie 
algebras.

For any $\germ \in \catprofgrp[W^{-1}]$, there is a sensible notion of a 
$\germ$-action $\germ \times X \to X$ on a topological space $X$: it means that 
some group $G$ whose germ is $\germ$ acts on $X$. If $G$ acts on $X$ with 
finitely many orbits, then any open subgroup of $G$ also acts on $X$ with 
finitely many orbits, and in this case we say that $\germ$ acts on $X$ with 
finitely many orbits. A $\germ$-action with finitely many orbits is a local 
structure in the following sense: if $X$ is partitioned into finitely many 
clopen subsets $X_1 \amalg \dotsb \amalg X_n$, then a $\germ$-action with 
finitely many orbits on $X$ determines and is determined by a $\germ$-action 
with finitely many orbits on each of the $X_i$.

In the body of the paper we do not work directly with $\catprofgrp[W^{-1}]$ but 
with equivalence classes of subgroups of the group of self-homeomorphisms of 
$X$.

Germs of group actions also appear in the work of Harman and Snowden 
\cite[Section~2.5]{HarSno24} on tensor categories.

\subsection{Finite tensor powers of bimodules}
\label{ssec:intro-bimod-fin}
When $\idemring$ is an associative algebra, a pair $\vsV,\vsW$ of 
$(\idemring,\idemring)$-bimodules can be tensored together to get a third 
$(\idemring,\idemring)$-bimodule, denoted $\vsV \tensor_{\idemring} \vsW$. The 
tensor product is spanned by pure tensors of the form $v \tensor_{\idemring} w$ 
and the fundamental relation is
\[
  v a \tensor_{\idemring} w = v \tensor_{\idemring} aw, \qquad \text{for all } 
  a \in \idemring, v \in \vsV, \text{ and } w \in \vsW.
\]
In \S~\ref{sec:bimod}, we will investigate the possibility of iterating this 
construction profinitely many times. Before explaining our point of view in 
\S~\ref{ssec:intro-bimod-profin}, we make some remarks about finite 
$\tensor_{\idemring}$-powers.

There are two basic considerations for $\tensor_{\idemring}$, that are not 
present for tensor products of vector spaces:
\begin{enumerate}
  \item It is almost always more appropriate to consider the derived version of 
    $\tensor_{\idemring}$. We will avoid this issue by considering only the 
    case where $\idemring$ is a semisimple algebra, when the derived and 
    underived version are the same. In fact, from \S~\ref{ssec:bimod-func-mcc} 
    onward, we work with $\idemring = \F_2 \times \dotsb \times \F_2$.
  \item The tensor product over $\idemring$ is not symmetric: usually $\vsV 
    \tensor_{\idemring} \vsW$ is different from $\vsW \tensor_{\idemring} 
    \vsV$.  Because of this, to make sense of the $\tensor_{\idemring}$-power 
    of $X$ copies of a single $(\idemring,\idemring)$-bimodule $\vsV$, even 
    when $X$ is finite, it is necessary to first specify the order that the 
    tensor products are taken.
\end{enumerate}

Before expanding on item (2), it is useful to make the observation that when 
$\idemring$ is semisimple, $\vsV \tensor_{\idemring} \vsW$ can be found in a 
natural way as a summand of $\vsV \tensor \vsW$ (obtained by forgetting the 
bimodule structures and taking the usual tensor product). The whole of $\vsV 
\tensor \vsW$ decomposes as a sum
\[
  \bigoplus_{\iota'_1,\iota''_1, \iota'_2,\iota''_2} \iota'_1 \vsV \iota''_1 
  \tensor \iota'_2 \vsW \iota''_2,
\]
where the $\iota$'s run through four independent copies of the set 
$\Spec(\idemring)$ of primitive central idempotents in $\idemring$. The 
bimodule tensor product $\vsV \tensor_\idemring \vsW$ is the sum of only those
$\iota'_1 \vsV \iota''_1 \tensor \iota'_2 \vsW \iota''_2$ where $\iota''_1 = 
\iota'_2$. Sometimes in physics this kind of summand is called a ``sector.''

The necessary ordering noted in item (2) is an extra structure on the set $X$ 
indexing the tensor factors:
\begin{enumerate}
  \item[(2L)] Write $X_L$ for a finite set $X$ together with a linear order. If 
    $X$ has $n$ elements and they are ordered $x_1 < x_2 < \dotsb < x_n$, then 
    an element $\bigtensor_{X_L} \vsV$ is a linear combination of pure tensors 
    like $v_{x_1} \tensor_{\idemring} v_{x_2} \tensor_{\idemring} \dotsb 
    \tensor_{\idemring} v_{x_n}$ subject to the same multilinear relations as 
    the usual tensor product, and to
    \begin{align*}
      v_{x_1}a \tensor_\idemring v_{x_2} \tensor_\idemring v_{x_3} 
      \tensor_\idemring \dotsb \tensor_\idemring v_{x_n} &= v_{x_1} 
      \tensor_\idemring av_{x_2} \tensor_\idemring v_{x_3} \tensor_\idemring 
      \dotsb \tensor_\idemring v_{x_n},\\
      v_{x_1} \tensor_\idemring v_{x_2}a \tensor_\idemring v_{x_3} 
      \tensor_\idemring \dotsb \tensor_\idemring v_{x_n} &= v_{x_1} 
      \tensor_\idemring v_{x_2} \tensor_\idemring a v_{x_3} \tensor_\idemring 
      \dotsb \tensor_\idemring v_{x_n},
    \end{align*}
    etc.
  \item[(2S)] Write $X_S$ for a finite set together with a cyclic order. If $X$ 
    has $n$ elements that are cyclically ordered $x_1 \to x_2 \to \dotsb \to 
    x_n \to x_1$, then $\bigtensor_{X_S} \vsV$ is spanned by the same pure 
    tensors as in (2L), subject to the same relations, and to one new relation
    \[
      av_{x_1} \tensor_{\idemring} v_{x_2}  \tensor_{\idemring} \dotsb  
      \tensor_{\idemring} v_{x_n} = v_{x_1} \tensor_{\idemring} v_{x_2}  
      \tensor_{\idemring} \dotsb  \tensor_{\idemring} v_{x_n}a.
    \]
\end{enumerate}
Again when $\idemring$ is semisimple we can find either 
$\tensor_{\idemring}$-power, $\bigtensor_{X_L} \vsV$ or $\bigtensor_{X_S} 
\vsV$, as a summand of the tensor power $\bigtensor_X \vsV$ obtained by 
forgetting the bimodule structure. If we write $\Spec(\idemring) \subset 
\idemring$ for the set of primitive central idempotents in $\idemring$, then 
the whole of $\bigtensor_X \vsV$ splits as a direct sum
\[
  \bigtensor_X \vsV = \bigoplus_{(\Spec(\idemring) \times \Spec(\idemring))^X} 
  \bigtensor_{x \in X} \iota_{x}' \vsV \iota_{x}''
\]
where the parts are indexed by assignments $x \mapsto (\iota_x', \iota_x'')$.  
The $\tensor_{\idemring}$-powers described in (2L) and (2S) are obtained by 
discarding those parts whose idempotents are not compatible with the linear or 
cyclic order.

Another standard  presentation of the tensor power (2S) invokes a construction 
that takes an $(\idemring,\idemring)$-bimodule $M$ to a vector space called 
$\HH_0(\idemring;M)$, which is the quotient of $M$ by the relations $am = ma$.  
The derived version of this construction is called the Hochschild chain 
complex, and $\HH_0$ is its $0$th homology. When $\idemring$ is semisimple, all 
the other homology groups of the Hochschild chain complex vanish.  The vector 
space described in item (2S) is the application of $\HH_0$ to the bimodule 
$\bigtensor_{X_L} \vsV$ defined in item (2L), so that
\begin{equation}
  \label{eq:intro-cyclic-linear}
  \bigtensor_{X_S} \vsV = \HH_0\left(\idemring;\bigtensor_{X_L} \vsV\right).
\end{equation}

\subsection{Profinite tensor powers of bimodules---the ``solenoidal sector''}
\label{ssec:intro-bimod-profin}
In \S~\ref{sec:vs} we will review the theory of the ``$X$-fold'' tensor power 
of a vector space $\vsV$ when $X$ is a finite set. We will discuss some 
obstacles to elaborating on this theory when $X$ is a profinite set. In case 
the ground field is $\F_2$ and a profinite $2$-group acts on $X$ with finitely 
many orbits, we find some tricks for getting around those obstacles.

In \S~\ref{sec:bimod}, we consider not a vector space $\vsV$ but a bimodule for  
an associative algebra. On a profinite set, it is easier to give a structure 
that resembles a cyclic order than it is to give a structure that resembles a 
linear order: the $X$-fold tensor product of bimodules that we will develop for 
profinite $X$ is the one summarized in item (2S) of 
\S~\ref{ssec:intro-bimod-fin}. As in \S~\ref{ssec:intro-vs-profin}, we obtain a 
reasonable definition only in a special situation:
\begin{itemize}
  \item $\idemring$ is a product of finitely many copies of the field $\F_2$, 
    and
  \item $X$ is acted on, with finitely many orbits, by a profinite $2$-group 
    $G$
\end{itemize}
Furthermore, we require a profinite version of the cyclic order $S$. We call it 
a ``solenoidal'' structure. Given this structure we define a magnetized and 
conditionally convergent version of (2S) that we denote by 
$\bigprofintensor{X_S} V$. There is no analogous equation to 
\eqref{eq:intro-cyclic-linear}.

A solenoidal structure is just a homeomorphism $S \colon X \to X$ that commutes 
with at least one open subgroup $\grpH \osubgrp G$, or said another way, that 
commutes with the germ $\germ$ acting on $X$. For each $\grpH$, such an $S$ 
induces a permutation of the finite quotient $\grpH \backslash X$. There is a 
natural way in which each of these permutations encodes a closed, oriented 
$1$-manifold, possibly with multiple components, that is subdivided into 
vertices and edges: each element of $\grpH \backslash X$ labels an edge of the 
subdivision, stretching between its starting and ending vertex, and $S$ takes 
each edge to the other edge incident with its ending vertex. Let us abuse the 
``subscript $S$'' notation and denote this subdivided $1$-manifold by $(\grpH 
\backslash X)_S$. In a similar way, $X$ itself labels the edges in a 
subdivision of a pro-$1$-manifold, i.e.\ in what is sometimes called a 
solenoid:
\begin{equation}
  \label{eq:intro-solenoid}
  X_S \text{ as a pro-$1$-manifold} := \varprojlim_{\grpH \osubgrp G} 
  \left((\grpH \backslash X)_S\text{ as a $1$-manifold}\right).
\end{equation}

\subsection{The staircase picture}
\label{ssec:intro-staircase}
The unusual feature of tensor products over $\F_2$ that makes our construction 
work is that an $\F_2$--vector space $\vsV$ is, in a rather canonical way, a 
subquotient of $\vsV \tensor_{\F_2} \vsV$. The symmetric tensors $\Sym^2(\vsV) 
\subset \vsV \tensor_{\F_2} \vsV$ project canonically onto $\vsV$ by the rule 
(which is only valid over $\F_2$)
\begin{equation*}
  \label{eq:intro-sqrt}
  v \tensor v \mapsto v \quad \text{and}\quad v \tensor w + w \tensor v \mapsto 
  0.
\end{equation*}
(We warn that in characteristic $2$ the natural invariant subspace of $\vsV 
\tensor \vsV$ is not canonically isomorphic to the natural ``coinvariant'' 
quotient of $\vsV \tensor \vsV$, and that the notation $\Sym^2(\vsV)$ is more 
commonly used for the quotient construction.)

This construction can be iterated:
\[
  \xymatrix{
    \ar[d] \vdots & \ar[d] \vdots & \ar[d] \vdots & \ar[d] \vdots & \iddots\\
    \ast \ar[d] \ar[r] & \ast \ar[d] \ar[r] & \Sym^2(\vsV^{\tensor 4}) \ar[r] 
    \ar[d] & \vsV^{\tensor 8}\\
    \ast \ar[r] \ar[d] & \Sym^2(\vsV^{\tensor 2}) \ar[r] \ar[d] & \vsV^{\tensor 
      4} \\
    \Sym^2(\vsV) \ar[r] \ar[d] & \vsV^{\tensor 2} \\
    \vsV
  }
\]
The unnamed groups in this diagram, denoted by asterisks, are fiber products.  
For $n > 0$, there are actually many natural inclusions of 
$\Sym^2(\vsV^{\tensor 2^n})$ into $\vsV^{\tensor 2^{n+1}}$, depending on how 
the indices for the factors in the smaller tensor product are identified with 
matched pairs of indices for the factors in the larger tensor product.  
Organizing these choices leads ultimately (and slightly more generally) to an 
indexing involving a profinite set $X$ with the action of a pro-$2$-group $G$, 
as in \S~\ref{ssec:intro-vs-profin}. The most natural formulation has columns 
of the ``staircase'' indexed by open subgroups $\colgrp \osubgrp G$, and the 
rows by open normal subgroups $\rowgrp \onsubgrp \colgrp$. Rows that are higher 
up, and columns that are further to the right, are indexed by smaller 
subgroups.

The limit $\varprojlim_{\rowgrp \onsubgrp \colgrp}$ of each column of the 
diagram is an $\F_2$--vector space with a natural, compact, topology. These 
limits are part of a direct system induced by the horizontal arrows in the 
diagram, and the colimit $\varinjlim_{\colgrp \osubgrp G}$ of this direct 
system is $\bigprofintensor{X}{\vsV}$.  Another way to put it is that 
$\bigprofintensor{X}{\vsV}$ has a pair of families of subspaces $F'_\colgrp$ 
and $F''_\colgrp$, indexed by open subgroups of $G$, and that 
$F'_\colgrp/F''_\colgrp$ is naturally identified with the finite tensor product 
$\bigtensor_{\colgrp \backslash X} \vsV$:
\begin{equation}
  \label{eq:intro-F/F-vs}
  F''_\colgrp \subset F'_\colgrp \subset \bigprofintensor{X}{\vsV}; \qquad 
  F'_\colgrp/F''_\colgrp \cong \bigtensor_{\colgrp \backslash X} \vsV.
\end{equation}

A similar staircase can be drawn in the setting of 
\S\S~\ref{ssec:intro-bimod-fin}--\ref{ssec:intro-bimod-profin}, when $\vsV$ is 
an $(\idemring,\idemring)$-bimodule: when $\idemring$ is a product of copies of 
$\F_2$, $\HH_0(\idemring,\vsV)$ is in a rather canonical way the subquotient of 
$\HH_0(\idemring,\vsV \tensor_{\idemring} \vsV)$, which in turn is a 
subquotient of $\HH_0(\idemring,(\vsV \tensor_{\idemring} \vsV) 
\tensor_{\idemring}(\vsV \tensor_{\idemring} \vsV))$ and so on. After using a 
$G$-solenoid $X_S$ to organize these subquotient relationships, one obtains a 
description of $\bigprofintensor{X_S}{\vsV}$ as an inverse limit followed by a 
direct limit. Another way to put it, say in the case that $X_S$ is 
connected,\footnote{Meaning that $S$ has one orbit on every $\colgrp \backslash 
  X$.  This condition implies that the pro-1-manifold \eqref{eq:intro-solenoid} 
  is connected in the sense of point-set topology, but note that no solenoid is 
  ever path connected.} is that $\bigprofintensor{X_S}{\vsV}$ has a similar 
structure as \eqref{eq:intro-F/F-vs}: a pair of subspaces $F''_\colgrp \subset 
F'_\colgrp$ for each open subgroup $\colgrp \osubgrp G$, with 
$F'_\colgrp/F''_\colgrp$ identified with $\HH_0$ of an ``in-a-line'' tensor 
power ((2L) of \S~\ref{ssec:intro-bimod-fin} of $\colgrp \backslash X$ many 
copies of $\vsV$:
\begin{equation}
  \label{eq:intro-F/F-bimod}
  F''_\colgrp \subset F'_\colgrp \subset \bigprofintensor{X_S}{\vsV};  \qquad 
  F'_\colgrp/F''_\colgrp \cong \bigtensor_{(\colgrp \backslash X)_S} \vsV \cong 
  \HH_0\left(\idemring,\bigtensor_{(\colgrp \backslash X)_L} \vsV\right).
\end{equation}

\subsection{An example}
\label{ssec:intro-ex}
We have been motivated by some computations in Heegaard Floer homology.

Ozsv\'ath and Szab\'o \cite{OzsSza04:HF} introduced invariants of a closed, 
oriented $3$-manifold $Y$, now called the Heegaard Floer homology groups of 
$Y$.
The Heegaard Floer homology of $Y$ is meant to be a model for the ground states 
of Seiberg--Witten theory on $\R \times Y$.

Ozsv\'ath and Szab\'o presented several versions of Heegaard Floer homology, 
called ``hat'', ``minus'', ``plus'', and ``infinity''. A version for manifolds 
with boundary, sutured Heegaard Floer homology, was introduced by Juh\'asz 
\cite{Juh06}.  We will describe our computations in the sutured framework.

Let $\knotL \subset S^3$ be the figure-eight knot. Let $(M,\partial M)$ be the 
exterior of $\knotL$, i.e.\ the complement in $S^3$ of a tubular neighborhood 
of $\knotL$. The boundary of $M$ is a torus, which we equip with two 
``sutures'' $\Gamma \subset \partial M$, i.e.\ two disjoint meridional circles 
in this torus.

Like the first homology of any knot exterior, the first homology of $M$ is an 
infinite cyclic group, for which we choose a generator: $H_1(M) \cong \Z$. For 
each finite-index subgroup $H \nsubgrp \Z$, let $M_H$ denote the covering space 
of $M$ whose deck group is $\Z/H$. The preimage of $\partial M$ in $M_H$ is the 
boundary of $M_H$, and is another connected torus that we denote by $\partial 
M_H$. The preimage of $\Gamma$ is another pair of circles that we denote by 
$\Gamma_H$. We are particularly interested in the tower
\begin{equation*}
  \dotsb \twoheadrightarrow (M_{4 \Z} ,\partial M_{4\Z},\Gamma_{4\Z})
  \twoheadrightarrow (M_{2 \Z} ,\partial M_{2\Z},\Gamma_{2\Z}) 
  \twoheadrightarrow (M,\partial M, \Gamma)
\end{equation*}
when $H$ runs through the subgroups of index $2^n$.

\begin{theorem*}
  Let $G \cong \Z_2$ be the $2$-adic completion of $H_1(M)$, let $X$ be an 
  affine copy of $G$, and let $S\colon X \to X$ be the ``obvious'' solenoidal 
  structure given by $S(x) = x+1$. There is a semisimple algebra $\idemring$ 
  and an $(\idemring,\idemring)$-bimodule $\vsV$ such that with
  \[
    F''_H \subset F'_H \subset \bigprofintensor{X_S} \vsV
  \]
  as in \eqref{eq:intro-F/F-bimod}, there are isomorphisms
  \[
    F'_H/F''_H \cong \left(
      \begin{array}{c}
        \text{the sutured Heegaard Floer}\\
        \text{homology of }(M_H,\partial M_H, \Gamma_H)
      \end{array}
    \right)
  \]
  where the left-hand side is as in \S~\ref{ssec:intro-staircase}.
\end{theorem*}

In particular, the isomorphisms show a subquotient relationship between 
Heegaard Floer homology of a $3$-manifold and Heegaard Floer homology of a 
double cover of that $3$-manifold. Other instances of this relationship were 
first observed by Hendricks \cite{Hen12}, and the relationship here is more 
specifically an instance of a result of \cite{LipTre16}.

The isomorphisms of the Theorem are natural, and to us suggest that 
$\bigprofintensor{X_S}\vsV$ is itself the sutured Heegaard Floer homology of \[
  (\varprojlim M_H, \varprojlim \partial M_H, \varprojlim \Gamma_H).
\] Here $\varprojlim M_H$ is a pro-$3$-manifold that $2$-adically approximates 
the Alexander cover of the knot exterior. Its boundary $\varprojlim \partial 
M_H$ is a pro-$2$-manifold that is the product of a solenoid with $S^1$, and 
$\varprojlim \Gamma_H$ is a disjoint union of two solenoids.  But there is no 
precedent definition for the sutured Heegaard Floer homology of such a 
pro-manifold and we do not formulate one in this paper.

\pagebreak[4]

\subsection{Questions}
\label{ssec:intro-q}

\subsubsection{}
\label{sssec:intro-q-other-rings}
Is there a reasonable extension of the magnetized and conditionally convergent 
tensor power of an $R$-module $\vsV$, if $R$ is a ring other than $\F_2$? We 
mentioned in \S~\ref{ssec:intro-vs-infin}--\ref{ssec:intro-vs-profin} that a 
basic obstacle is making sense of the product $\prod_{x \in X} c_x$ of an 
$X$-indexed collection of scalars $c_x$.  When $R = \Z/2^n$, there is a 
seemingly reasonable way to do this when the $c_x$'s depend in a locally 
constant fashion on $x$, or more particularly when they are constant on 
finitely many $G$-orbits for a pro-$2$-group $G$ acting on $X$: we can set
\[
  \prod_{x \in X} c_x = \begin{cases}
    1 & \text{if every $c_x$ is odd}, \\
    0 & \text{if any $c_x$ is even}.
  \end{cases}
\]
The justification for this formula is that, if the $c_x$'s are ``magnetized'' 
i.e.\ constant on $G$-orbits, then $\prod_{x \in X} c_x = \prod_{O} 
(c_O)^{\card(O)}$, where $O$ runs through the orbits. If you regard $\card(O)$ 
as an infinitely large power of $2$, then $c^{\card(O)}$ is $1$ or $0$ 
according to whether $c$ is odd or even.  If it works for $\Z/2^n$-modules, one 
can next think of getting it to work for $\Z_2$-modules and $\Q_2$--vector 
spaces.

\subsubsection{}
\label{sssec:intro-q-staircase}
The staircase picture \S~\ref{ssec:intro-staircase} suggests another 
construction for perfect rings of characteristic $2$. If $R$ is a ring of 
characteristic $2$, and $\vsV$ is a free $R$-module, in general it is not 
$\vsV$ but $F(\vsV)$ that appears as a subquotient of $\vsV \tensor_R \vsV$, 
where $F$ denotes the ``Frobenius twist'' that has the same underlying abelian 
group as $\vsV$, but whose $R$-module structure is via $r \cdot_F v := r^2 v$.  
If $R$ is perfect, then $F$ can be inverted, and $\vsV$ appears as a 
subquotient of $F^{-1} (\vsV \tensor \vsV)$, leading to a staircase picture 
where $F^{-n}$ has been applied to the $n$th row of the original staircase 
picture, and a kind of tensor power can be obtained by taking the direct limit 
of the inverse limit of the staircase:
\[
  \xymatrix{
    \ar[d] \vdots & \ar[d] \vdots & \ar[d] \vdots & \iddots \\
    \ast \ar[r] \ar[d] & F^{-2}(\Sym^2(\vsV^{\tensor 2})) \ar[r] \ar[d] & 
    F^{-2}(\vsV^{\tensor 4}) \\
    F^{-1}(\Sym^2(\vsV)) \ar[r] \ar[d] & F^{-1}(\vsV^{\tensor 2}) \\
    \vsV
  }
\]

\subsubsection{}
\label{sssec:intro-q-C2}
Can the two ideas \S~\ref{sssec:intro-q-other-rings} and 
\S~\ref{sssec:intro-q-staircase} be combined, perhaps, to obtain a definition 
of a tensor power of $\C_2$--vector spaces, or of modules over a more general 
perfectoid ring?

\subsubsection{}
\label{sssec:intro-q-hom-alg}
Can any homological algebra be done with $\bigprofintensor{X} \vsV$?  Suppose 
that $\vsV$ is a chain complex. One problem is that if $\vsV$ has any element 
of nonzero degree, then there are magnetized pure tensors in 
$\bigprofintensor{X} \vsV$ that have apparently ``infinite'' degree.  A second, 
probably more serious, problem, is that the formula for 
$\partial(\bigtensor_{X} v_x)$ required by the Leibniz rule is a sum of 
non-magnetized values, even if $\bigtensor_{X} v_x$ is magnetized.

\subsubsection{}
\label{sssec:intro-q-hf}
Is there a definition of the Heegaard Floer homology of a pro-$3$-manifold?

\subsection*{Acknowledgments} The authors thank John Baldwin and Robert 
Lipshitz for helpful conversations. CMMW thanks the number theory group at 
Dartmouth College in Spring 2022, including Eran Assaf, Asher Auel, Avi 
Kulkarni, Jack Petok, Ciaran Schembri, and John Voight, for patiently teaching 
him things $p$-adic in an inspiring seminar series.
CMMW acknowledges the support of the Natural Sciences and Engineering Research 
Council of Canada (NSERC), RGPIN-2023-05123.

\section{Tensor powers of vector spaces}
\label{sec:vs}

Throughout this section, $R$ denotes a commutative ring and $\vsV$ denotes a 
free $R$-module of finite rank. In the later parts of this section
(\S\S~\ref{ssec:vs-F2-inf-prods-sums}--\ref{ssec:vs-sector}), we assume that $R 
= \F_2$.  We use $X$ to stand for either a finite set, or a profinite set with 
a ``germ action'' (\S~\ref{ssec:vs-profin-sets}). A germ action is an 
equivalence class of group actions on $X$, we use $G$ to stand for one of the 
groups of this equivalence class. These groups are always profinite, and from 
\S~\ref{ssec:vs-F2-inf-prods-sums} on, they have an open subgroup of $2$-power 
order.

\subsection{Profinite sets with 
  \texorpdfstring{$\germ$-structures}{g-structures}}
\label{ssec:vs-profin-sets}
Let $X$ be a Hausdorff topological space. Let $\Homeo(X)$ denote the group of 
self-homeomorphisms of $X$ with its compact-open topology. Consider the set 
$\groups$ of all closed subgroups $G \subset \Homeo(X)$ that have both of the 
properties
\begin{enumerate}
  \item $G$ is profinite; and
  \item $G$ acts on $X$ with finitely many orbits.
\end{enumerate}
If $\groups$ is not empty, then $X$ has the topology of a profinite set: it is 
the disjoint union of its finitely many $G$-orbits, each of which is 
homeomorphic to $G/H$ for a closed subgroup $H \subgrp G$.

Two groups $G_1, G_2 \in \groups$ are ``commensurable'' if $G_1 \cap G_2$ is an 
open subgroup of both $G_1$ and of $G_2$---in that case $G_1 \cap G_2$ is 
another element of $\groups$. Commensurability is an equivalence relation on 
$\groups$ and we call an equivalence class for this equivalence relation a 
``germ action'' on $X$. 

If $G$ is a group in $\groups$ in the equivalence class $\germ$, we write $G 
\supset \germ$ and we sometimes say that $\germ$ is the germ of $G$. The 
inclusion relation makes the set of $G$ whose germ is $\germ$ into a filtered 
poset, in the sense that any two elements of the poset have a common lower 
bound.

If one of the $G \supset \germ$ has pro-$p$-power order, or equivalently if 
every $G \supset \germ$ has an open subgroup of pro-$p$-power order, we say 
that $\germ$ has pro-$p$-power order. Later on 
(\S~\ref{ssec:vs-F2-inf-prods-sums}), we will assume that $\germ$ has 
pro-$2$-power order.

It is useful to think of $\germ$ as an object that acts on $X$, as well as on 
other sets and spaces derived from $X$. If $Y$ is a topological space (often a 
discrete space, below), we will say that $\germ$ acts continuously on $Y$ if 
there is a $G \supset \germ$ that acts continuously on $Y$.

\subsection{Magnetized pure tensors}
\label{ssec:vs-mag}
If $B$ and $X$ are finite sets, let $B^X$ denote the set of functions from $X$ 
to $B$. If $B$ is a basis for a free $R$-module $\vsV$, then the finite set 
$B^X$ indexes the induced basis on the tensor power $\bigtensor_X \vsV$, i.e.\ 
the basis of pure tensors whose tensor factors are from $B$.

If $B$ is a finite set and $X$ is a \emph{profinite} set, then endow $B$ with 
the discrete topology and let $B^X$ denote the set of \emph{continuous} 
functions from $X$ to $B$. If $B$ is a basis for $\vsV$, then we call $B^X$ the 
set of ``magnetized pure tensors'' with tensor factors from $B$.

If $X$ has a germ action $\germ$, then every continuous function $f \colon X 
\to B$ is constant on $G$-orbits for every sufficiently small $G \supset 
\germ$.  In other words, if we write $G \backslash X$ for the finite set of 
orbits and $B^{G\backslash X} \subset \Fun(X,B)$ for the set of functions that 
factor through $G \backslash X$, then $B^X = \bigcup_{G \supset \germ} B^{G 
  \backslash X}$.
An equivalent description of $B^X$ uses the $\germ$-action on the set of all 
functions $X \to B$. Specifically, if $\Fun(X,B)$ denotes the set of all (not 
necessarily continuous) functions, and $G \supset \germ$, then we let $g \in G$ 
act on $f \in \Fun(X,B)$ by the formula
\[
  (g,f) \mapsto g \cdot f, \qquad (g \cdot f)(x):= f(g^{-1}(x)).
\]
Then $B^{G\backslash X} = \Fun(X,B)^G$, and
\begin{equation}
  \label{eq:BX-smooth}
  \basisB^X = \bigcup_{G \supset \germ} \Fun(X,\basisB)^{G}.
\end{equation}

\subsection{Conditionally convergent tensors}
\label{ssec:vs-cc}
A general tensor is a linear combination of pure tensors. If $P$ is any set, 
then the most general kind of formal $R$-linear combination of elements of $P$ 
is an element of the direct product $\prod_P R$, or equivalently an element of 
the set of functions $\Fun(P,R)$. The subset $\bigoplus_P R$ of finite linear 
combinations is the set $\Fun(P,R)^f$ of those functions $P \to R$ with finite 
support. When $P = B^X$ is the set of magnetized pure tensors with tensor 
factors from $B$, and $\germ$ is a germ action on $X$, then we can define an 
intermediate set
\begin{equation}
  \label{eq:fun-BX-R}
  \Fun(\basisB^X,R)^f \subset \Fun(\basisB^X,R)^{\cc} \subset \Fun(\basisB^X,R)
\end{equation}
that we call ``conditionally convergent'' magnetized tensors. Specifically, if 
$G \supset \germ$, then $g \in G$ acts on $\phi \in \Fun(\basisB^X,R)$ via
\[
  (g,\phi) \mapsto g \cdot \phi, \qquad (g\cdot \phi)(f) = \phi(g^{-1} \cdot 
  f).
\]
Another useful formula is $((g \cdot \phi)(f))(x) = \phi(f(gx))$. Then 
\begin{equation}
  \label{eq:fun-BX-R-cc}
  \Fun(\basisB^X,R)^\cc := \bigcup_{G \supset \germ} \Fun(\basisB^X,R)^G.
\end{equation}
In a way, this formula is similar to \eqref{eq:BX-smooth}.  The containment 
$\Fun(\basisB^X,R)^{\cc} \subset \Fun(\basisB^X,R)$ of \eqref{eq:fun-BX-R} is 
by definition. The containment $\Fun(\basisB^X,R)^f \subset 
\Fun(\basisB^X,R)^{\cc}$ is a consequence of the fact that any finite subset of 
$\basisB^X$ is contained in $\basisB^{G \backslash X}$ for some $G \supset 
\germ$: if $\phi \in \Fun(\basisB^X,R)$ is supported on such a finite subset, 
then $\phi$ also belongs to $\Fun(\basisB^X,R)^{G}$.

It is somewhat premature to call these tensors ``conditionally convergent'': 
later (\S~\ref{ssec:vs-F2-inf-prods-sums}), the condition will allow us to 
evaluate certain infinite sums, but only when $\germ$ is $2$-adic and $R = 
\F_2$. In the representation theory of $p$-adic groups, the subset 
$\Fun(\basisB^X,R)^{\cc} \subset \Fun(\basisB^X,R)$ would be called the set of 
``algebraic'' or ``smooth'' vectors \cite{BerZel76}.

\subsection{Change of basis}
\label{ssec:vs-change-of-basis}
If $B$ is a basis for the $R$-module $\vsV$, then we would like to treat 
$\Fun(\basisB^X,R)^{\cc}$ as the $X$-fold tensor power of $\vsV$. An immediate 
question is whether $\Fun(\basisB^X,R)^{\cc}$ depends on the basis $\basisB$ of 
$\vsV$ used to define it.

Let $\basisC$ be a second basis for $\vsV$. We would like the change of basis 
matrix to determine an isomorphism
\[
  \Fun(\basisB^X,R)^{\cc} \cong \Fun(\basisC^X,R)^{\cc}.
\]
Over \S\S~\ref{ssec:vs-cat-fin-matrices}--\ref{ssec:vs-func-mcc-2}, we will 
show that $\Fun^{\cc}$ has this and other functorial properties as long as $R$ 
is the field with $2$ elements and $\germ$ is $2$-adic.
In the more trivial case when $X$ is finite, $\germ$ plays no role and no 
condition on $R$ is necessary. In \S~\ref{ssec:vs-cat-fin-matrices} and 
\S~\ref{ssec:vs-func-fin-tensor}, we set up some notation for dealing with 
finite matrices $M$ and their tensor powers $M^{\tensor X}$ for general $R$ and 
finite $X$. The expression defining $M^{\tensor X}$ can be written down for 
infinite $X$, but it involves an infinite product and an infinite sum. In 
\S~\ref{ssec:vs-F2-inf-prods-sums} we will explain the role of our assumptions 
on $X$, $\germ$, and $R$ in making sense of this product and sum.

Once the germ action is fixed and $R = \F_2$, we suppress $\basisB$ from the 
notation and set
\[
  \bigprofintensor{X} \vsV := \Fun(\basisB^X,\F_2)^{\cc}.
\]
We will see that it is a functor from finite-dimensional $\F_2$--vector space 
to the category of all $\F_2$--vector spaces.

\subsection{Category whose morphisms are finite matrices}
\label{ssec:vs-cat-fin-matrices}
We will use $\Mat^R$ denote the natural category whose objects are finite sets 
and whose morphisms are matrices with entries from $R$. We will use the symbols 
$\basisB$, $\basisC$, \dots to denote finite sets; then $\Mat^R$ is the 
category whose objects are finite sets and whose morphisms 
$\basisB\xrightarrow{M}\basisC$ are functions
\begin{equation}
  \label{eq:M-CxB-R}
  M \colon \basisC \times \basisB \to R.
\end{equation}
We will call any such function a ``matrix.'' In \eqref{eq:M-CxB-R}, $\basisC$ 
is indexing the rows and $\basisB$ is indexing the columns of a matrix whose 
entries are the values $M(c,b)$ of $M$.

The composition law
\[
  \left[\basisB \xrightarrow{M} \basisC \xrightarrow{N} \basisD \right] \mapsto 
  \left[ \basisB \xrightarrow{NM} \basisD \right]
\]
is given by the usual formula for matrix multiplication (or convolution of 
kernels):
\begin{equation}
  \label{eq:NM-d-b}
  (NM)(d,b) = \sum_{c \in \basisC} N(d,c)M(c,b).
\end{equation}
The set of functions $\Fun(\basisB,R)$ is a free $R$-module with a natural 
basis indexed by $\basisB$, and the assignment \begin{equation}
  \label{eq:Fun-B-R}
  \basisB \mapsto \Fun(\basisB,R)
\end{equation}
is covariantly functorial for matrices $\basisB \xrightarrow{M} \basisC$ by the 
rule:
\begin{gather*}
  f \in \Fun(\basisB,R), \quad f \mapsto Mf \in \Fun(\basisC,R);\\
  (Mf)(c) = \sum_{b \in \basisB} M(c,b)f(b).
\end{gather*}
The functor \eqref{eq:Fun-B-R} is an equivalence between $\Mat^R$ and the full 
subcategory of $\Mod(R)$ spanned by the free $R$-modules of finite rank.

\subsection{Finite tensor powers as a functor}
\label{ssec:vs-func-fin-tensor}
Let $X$ be a finite set and, as in \S~\ref{ssec:vs-mag}, let $\basisB^X$ denote 
the set of functions from $X$ to $\basisB$.
Like \eqref{eq:Fun-B-R}, the assignment
\begin{equation}
  \label{eq:Fun-BX-R}
  \basisB \mapsto \Fun(\basisB^X,R)
\end{equation}
is also functorial for matrices, in a sense that we now explain. Since $X$ is 
finite, $\basisB^X$ is also finite and \eqref{eq:Fun-BX-R} also takes values in 
the category of free $R$-modules of finite rank. 

Given $\basisB \xrightarrow{M} \basisC$, define a map $M^{\tensor X} \colon 
\Fun(\basisB^X,R) \to \Fun(\basisC^X,R)$ by the rule
\begin{equation}
  \label{eq:MX-fin}
  (M^{\tensor X} \phi)(g) = \sum_{f \in \basisB^X} \phi(f) \prod_{x \in X} 
  M(g(x),f(x)).
\end{equation}
If $X$ has $n$ elements and we order them arbitrarily $x_1,x_2,\dotsc,x_n$, 
then we define a map
\begin{equation}
  \label{eq:Fun-B-R-in-ord}
  \Fun(\basisB^X,R) \to \Fun(\basisB,R) \tensor \dotsb \tensor \Fun(\basisB,R)
\end{equation}
to the tensor product of $n$ copies of $\Fun(\basisB,R)$. When $\phi \in 
\Fun(\basisB^X,R)$, \eqref{eq:Fun-B-R-in-ord} sends $\phi$ to
\[
  \sum_f \phi(f)\cdot \left(\delta_{f(x_1)} \tensor \delta_{f(x_2)} \tensor 
    \dotsb \tensor \delta_{f(x_n)}\right).
\]
where $\delta_b \colon \basisB \to R$ takes $b$ to $1$ and every other element 
of $\basisB$ to zero.
The map \eqref{eq:Fun-B-R-in-ord} is a natural isomorphism making the following 
triangle of functors commute:
\[
  \xymatrix{
    \Mat^R \ar[rr]^{\basisB \mapsto \Fun(\basisB,R)} \ar[dr]_-{\basisB \mapsto 
      \Fun(\basisB^X,R);\, M \mapsto M^{\tensor X} \qquad} & & 
    \Mod(R)\ar[dl]^{\qquad \vsV \mapsto \bigtensor_{i=1}^n \vsV} \\
    & \Mod(R)
  }
\]
When we wish to ignore the ordering of $X$, we may denote the codomain of 
\eqref{eq:Fun-B-R-in-ord} by $\bigtensor_X \Fun(\basisB,R)$. 

\subsection{Infinite products and infinite sums in \texorpdfstring{$\F_2$}{the 
    field of two elements}}
\label{ssec:vs-F2-inf-prods-sums}
If $X$ is profinite, we would like to define
\[
  M^{\tensor X} \colon Fun(\basisB^X,R)^{\cc} \to \Fun(\basisC^X,R)^{\cc}
\]
by the same formula \eqref{eq:MX-fin}. We encounter two problems: (1) it is not 
clear how to interpret the product of all the scalars $M(g(x),f(x))$, indexed 
by the uncountable set $X$, and (2) even if it were, it is not clear how to 
interpret the sum of all the scalars $\phi(f) \prod M(g(x),f(x))$, indexed by 
the infinite discrete set of $\phi \in \basisB^X$.

There is a peculiar way around problem (1) when $R$ is the field with $2$ 
elements, and a peculiar way around problem (2) when in addition $X$ is acted 
on, with finitely many orbits, by a pro-$2$-group. 

The rule for (2) depends only on the germ $\germ$ of the pro-$2$-group that is 
acting.

\begin{proposition*}
  Suppose that $R$ has characteristic $2$, that $\germ$ has pro-2-power order, 
  and that $F$ is a discrete set with a continuous $\germ$-action.  Recall this 
  means that some $G \supset \germ$ acts continuously on $F$.  Suppose also 
  that $F$ has the following property:
  \begin{quote}
    For all $G \supset \germ$ sufficiently small, the fixed-point set $F^G$ is 
    finite.
  \end{quote}
  Then for any given $\phi \in \Fun(F,R)^{\cc}$, the formula
  \begin{equation}
    \label{eq:smith-sum}
    \phi \mapsto \Sigma(\phi) := \text{the finite sum } \sum_{f \in F^G} 
    \phi(f)
  \end{equation}
  takes the same value for all $G \supset \germ$ that fixes $\phi$, i.e.\ it 
  defines an $R$-linear map
  \[
    \Sigma \colon \Fun(F,R)^{\cc} := \bigcup_{G \supset \germ} \Fun(F,R)^{G} 
    \to R.
  \]
\end{proposition*}

\begin{proof}
  If $\phi$ is both $G_1$-fixed and $G_2$-fixed, then
  \[
    \sum_{f \in F^{G_1}} \phi(f) = \sum_{f \in F^{G_2}} \phi(f)
  \]
  because both sums are equal to $\sum_{f \in F^N} \phi(f)$ for any open 
  subgroup $N$ that is normal in both $G_1$ and $G_2$. The reason is that (as 
  both $G_1/N$ and $G_2/N$ are finite $2$-groups) there is always a $g_1 \in 
  G_1 \setminus N$ such that $g_1^2 \in N$, and then $\phi(f)$ and $\phi(g_1 
  f)$ cancel each other in
  \[
    \sum_{f \in F^N \setminus F^{G_1}} \phi(f).
  \]
  Similarly, $\phi(f)$ and $\phi(g_2 f)$ cancel each other in $\sum_{f \in F^N 
    - F^{G_2}} \phi(f)$.
\end{proof}

Because of the Proposition above, we denote $\Sigma (\phi)$ by
\[
  \sum_{f \in F} \phi (f).
\]
In particular, we can take $F = \basisB^X$, the set of magnetized pure tensors: 
$(\basisB^X)^{G} \cong \basisB^{G \backslash X}$ is a finite set for all $G 
\supset \germ$, because $B$ and $G \backslash X$ are finite sets.

The rule for (1) is more crude: if $X$ is an arbitrary set (profinite or not), 
and $(m_x)_{x \in X}$ is a set of mod $2$ integers (i.e.\ $m_x \in \F_2$) 
indexed by $x \in X$, define
\begin{equation}
  \label{eq:inf-prod}
  \prod_{x \in X} m_x = \begin{cases}
    1 & \text{if $m_x = 1$ for all $x$}, \\
    0 & \text{if there is some $x$ such that $m_x$ = 0}.
  \end{cases}
\end{equation}
The conditions on the scalars $(m_x \in R)_{x \in X}$ of \eqref{eq:inf-prod} 
are mutually exclusive if and only if $R = \F_2$.

Unlike the rule for evaluating infinite sums, this rule for evaluating infinite 
products does not work when $\F_2$ is replaced by a larger ring $R$ of 
characteristic $2$. In what follows we will assume that $R = \F_2$.

\subsection{Functoriality of \texorpdfstring{$\Fun((-)^X,\F_2)^{\cc}$}{the 
    profinite tensor product}, I}
\label{ssec:vs-func-mcc-1}
Let $M \colon \graphC \times \graphB \to \F_2$. We will make use of the map
\begin{equation*}
  \label{eq:smith-sum-BX}
  \Sigma \colon \Fun(\basisB^X,\F_2)^{\cc} \to \F_2
\end{equation*}
we have defined in \eqref{eq:smith-sum}, with $F = \basisB^X$, to define the 
map $M^{\tensor X} \phi$ in \eqref{eq:MX-fin}. To do so, we have to verify that
\begin{enumerate}
  \item for every $g \in \basisC^X$, the function $f \mapsto \phi(f) \prod_x 
    M(g(x),f(x))$ belongs to $\Fun(\basisB^X,\F_2)^{\cc}$ and
  \item the resulting function $g \mapsto (M^{\tensor X} \phi)(g)$ is an 
    element of $\Fun(\basisC^X,\F_2)^{\cc}$
\end{enumerate}

The Proposition below states this more precisely; it proves that the assignment 
$M \mapsto M^{\tensor X}$ is well-defined for every finite matrix $[\basisB 
\xrightarrow{M} \basisC]$. In the next subsection we will verify that 
$N^{\tensor X} \circ M^{\tensor X} = (NM)^{\tensor X}$.

\begin{proposition*}
  Let $\basisB$ and $\basisC$ be finite sets, and let $M \colon \basisC \times 
  \basisB \to \F_2$ be a matrix. If $\phi \in \Fun(\basisB^X,\F_2)^{\cc}$, then 
  for every $g \in \basisC^X$, the function
  \[
    f \mapsto \phi(f) \prod_{x \in X} M(g(x),f(x))
  \]
  is fixed by some $G \supset \germ$, and hence is also in 
  $\Fun(\basisB^X,\F_2)^{\cc}$.  Furthermore, the function
  $
  M^{\tensor X} \phi \colon \basisC^X \to \F_2  $ defined by
  \[
    (M^{\tensor X} \phi)(g) := \sum_{f \in \basisB^X} \phi(f) \prod_{x \in X} 
    M(g(x),f(x))
  \]
  belongs to $\Fun(\basisC^X,\F_2)^{\cc}$.
\end{proposition*}

\begin{proof}
  Since $g \colon X \to \basisC$ is continuous, it is fixed by some $G'_g 
  \supset \germ$. Since $\phi \in \Fun(\basisB^X,\F_2)^{\cc}$, it is fixed by 
  some $G''_\phi \supset \germ$. The intersection $G = G_{g, \phi} := G'_g \cap 
  G''_\phi$ is another group whose germ is $\germ$. We claim that the map
  \[
    f \mapsto \phi(f) \prod_{x \in X} M(g(x),f(x))
  \]
  is equal to its translate by $h$ for every $h \in G_{g, \phi}$. To spell this 
  out, let us write $\sigma_g(f)$ for this function:
  \[
    \sigma_g(f) = \phi(f) \prod_{x \in X} M(g(x),f(x)) = \begin{cases}
      \phi(f) &\text{if } M(g(x),f(x)) = 1 \text{ for all } x \in X,\\
      0 & \text{otherwise}.
    \end{cases}
  \]
  Then $h \sigma_g$ is given by
  \[
    h \sigma_g(f) = \sigma_g(h^{-1}f) = \begin{cases}
      \phi(h^{-1} f) & \text{if } M(g(x),f(hx)) = 1  \text{ for all } x \in 
      X,\\
      0 & \text{otherwise}.
    \end{cases}
  \]
  (In the condition for the first case, we have used $(h^{-1} f)(x):= f(hx)$.) 
  Since $\phi$ is $G''_\phi$-invariant and $h \in G''_\phi$, we have 
  $\phi(h^{-1} f) = \phi(f)$.  Since $g(x)$ is $G'_g$-invariant and $h \in 
  G'_g$, we have $g(x) = g(hx)$.  Since $h$ is a bijection from $X$ to $X$, the 
  condition that $M(g(hx),f(hx)) =1$ for all $x \in X$ is equivalent to the 
  condition that $M(g(x),f(x)) = 1$ for all $x \in X$.  Thus, $h \sigma_g(f) = 
  \sigma_g(f)$.

  Thus, the map $g \mapsto \phi(f) \prod_{x \in X} M (g (x), f (x))$ is a 
  function from $C^X$ to $\Fun (\basisB^X, \F_2)^{\cc}$. It follows from the 
  Proposition in \S~\ref{ssec:vs-F2-inf-prods-sums} that
  \[
    \phi(f) \prod_{x \in X} M (g (x), f (x)) \mapsto M^{\tensor X} \phi(g) = 
    \sum_{f \in B^X} \phi(f) \prod_{x \in X} M (g (x), f (x))
  \]
  is a well-defined function from $\Fun(\basisB^X,\F_2)^{\cc}$ to $\F_2$.  
  Composing the two, we have that $g \mapsto M^{\tensor X} \phi(g)$ is a 
  well-defined function from $C^X$ to $\F_2$, i.e.\ $M^{\tensor X} \phi \in 
  \Fun (C^X, \F_2)$.  To complete the proof, we will show that if $\phi \in 
  \Fun(\basisB^X,\F_2)^{G''_\phi}$, then $M^{\tensor X} \phi$ belongs to 
  $\Fun(\basisC^X,\F_2)^{G''_\phi}$. Indeed, when $h \in G''_\phi$,
  \[
    \begin{split}
      (h (M^{\tensor X} \phi))(g) & = (M^{\tensor X} \phi)(h^{-1} g) \\
      & =  \sum_{f \in \basisB^X} \phi(f) \prod_{x \in X} M(g(hx),f(x)) \\
      & =  \sum_{f \in \basisB^X} \phi(h^{-1} f) \prod_{x \in X} M(g(hx),f(hx)) 
      \\
      & =  \sum_{f \in \basisB^X} \phi(h^{-1} f) \prod_{x \in X} M(g(x),f(x)) 
      \\
      & =  \sum_{f \in \basisB^X} \phi(f) \prod_{x \in X} M(g(x),f(x)),
    \end{split}
  \]
  which is the definition of $M^{\tensor X} \phi (g)$. In the third step, we 
  have reindexed the sum over $f$ to be a sum over $h^{-1} f$. In the fourth 
  step, we have reindexed the product over $x$ to be a product over $h^{-1} x$.  
  In the last step, we have used $\phi(f) = \phi(h^{-1} f)$.
\end{proof}

\subsection{Functoriality of \texorpdfstring{$\Fun((-)^X,\F_2)^{\cc}$}{the 
    profinite tensor product}, II}
\label{ssec:vs-func-mcc-2}
Let $\left[\basisB \xrightarrow{M} \basisC\right]$ and $\left[\basisC 
  \xrightarrow{N} \basisD\right]$ be morphisms in $\Mat^{\F_2}$ and let $\left[ 
  \basisB \xrightarrow{NM} \basisD\right]$ be their composite. That is,
\[
  NM(d,b) = \sum_{c \in \basisC} N(d,c)M(c,b).
\]
By the Proposition of \S~\ref{ssec:vs-func-mcc-1}, they induce linear maps
\begin{equation}
  \label{eq:Fun-cc-MX-NX}
  \Fun(\basisB^X,\F_2)^{\cc} \xrightarrow{M^{\tensor 
      X}}\Fun(\basisC^X,\F_2)^{\cc} \xrightarrow{N^{\tensor X}} 
  \Fun(\basisD^X,\F_2)^\cc
\end{equation}
and
\begin{equation}
  \label{eq:Fun-CC-NMX}
  \Fun(\basisB^X,\F_2)^{\cc} \xrightarrow{(NM)^{\tensor X}} 
  \Fun(\basisD^X,\F_2)^\cc.
\end{equation}

\subsection*{Proposition} The composition of the maps in  
\eqref{eq:Fun-cc-MX-NX} is equal to \eqref{eq:Fun-CC-NMX}.

\begin{proof}
  Let $\phi \in \Fun(\basisB^X,\F_2)^{\cc}$. We will derive equal expressions 
  for $N^{\tensor X}(M^{\tensor X}(\phi))$ and $(NM)^{\tensor X}(\phi)$.

  First,
  the function $N^{\tensor X}(M^{\tensor X}(\phi)) \colon \basisD^X \to \F_2$ 
  is given by
  \begin{equation}
    \label{eq:NX-MX-on-h}
    h \mapsto \sum_{g \in \basisC^X} (M^{\tensor X}\phi(g)) \prod_{x \in X} 
    N(h(x),g(x)).
  \end{equation}
  Since $M^{\tensor X}\phi \in \Fun (\basisC^X, \F_2)^{\cc}$, the first 
  statement of the Proposition of \S~\ref{ssec:vs-func-mcc-1} says that the 
  function
  \[
    g\mapsto  (M^{\tensor X}\phi(g)) \prod_{x \in X} N(h(x),g(x))
  \]
  is constant on $G$-orbits for some  $G = G_{h,M^{\tensor X} \phi} = 
  G_{h,\phi}$ whose germ is $\germ$.  By definition, the infinite sum in 
  \eqref{eq:NX-MX-on-h} is equal to the finite sum
  \[
    \sum_{g \in (\basisC^X)^{G_{h,\phi}}} (M^{\tensor X}\phi(g)) \prod_{x \in 
      X} N(h(x),g(x)).
  \]
  Meanwhile, the $M^{\tensor X}\phi(g)$ factor of each summand is given by
  \begin{equation}
    \label{eq:MX-phi-g}
    \sum_{f \in \basisB^X} \phi(f) \prod_{x \in X} M(g(x),f(x))
    =
    \sum_{f \in (\basisB^X)^{G'_{g,\phi}}} \phi(f) \prod_{x \in X} 
    M(g(x),f(x)).
  \end{equation}
  for some $G'_{g,\phi} \supset \germ$, again by the Proposition of 
  \S~\ref{ssec:vs-func-mcc-1}. Set $G'_{h,\phi} := \bigcap_{g \in 
    (\basisC^X)^{G_{h,\phi}}} G'_{g,\phi}$. Since $ (\basisC^X)^{G_{h,\phi}}$ 
  is finite, $G'_{h,\phi}$ is open in every $G'_{g,\phi}$ and has the same germ 
  $\germ$.  Therefore, \eqref{eq:MX-phi-g} is equal to the finite sum
  \[
    \sum_{f \in (\basisB^X)^{G'_{h,\phi}}} \phi(f) \prod_{x \in X} 
    M(g(x),f(x)),
  \]
  and the index set $(\basisB^X)^{G'_{h,\phi}}$ is independent of $g$.  
  Altogether, we conclude that the value of $N^{\tensor X}(M^{\tensor 
    X}(\phi))$ on $h$ is
  \[
    \begin{split}
      & \phantom{{}=} \sum_{g \in (\basisC^X)^{G_{h,\phi}}} \left(\sum_{f \in 
          (\basisB^X)^{G'_{h,\phi}}} \phi(f) \prod_{x \in X} 
        M(g(x),f(x))\right) \prod_{x \in X} N(h(x),g(x))\\
      & = \sum_{g \in (\basisC^X)^{G_{h,\phi}}} \sum_{f \in 
        (\basisB^X)^{G'_{h,\phi}}} \phi (f) \prod_{x \in X} N(h(x),g(x)) 
      M(g(x),f(x)).
    \end{split}
  \]
  Separately, we know that $h \colon X \to D$ is constant on the orbits of 
  another group $G''_h \supset \germ$. Set $H_{h,\phi} := G_{h,\phi} \cap 
  G'_{h,\phi} \cap G''_h$; then $H_{h,\phi}$ also has the germ $\germ$.  By the 
  Proposition of \S~\ref{ssec:vs-F2-inf-prods-sums}, we may replace the index 
  sets $(C^X)^{G_{h,\phi}}$ and $(B^X)^{G'_{h,\phi}}$ in the two finite sums 
  above by $(C^X)^{H_{h,\phi}}$ and $(B^X)^{H_{h,\phi}}$. Then, since $f$, $g$, 
  and $h$ are all fixed by $H_{h,\phi}$, we may replace the infinite index set 
  $X$ in the product by the finite index set $\lmod{X}{H_{h,\phi}}$. Combining 
  everything, we see that the sum in \eqref{eq:NX-MX-on-h} can be written as
  \[
    \sum_{g \in (\basisC^X)^{H_{h,\phi}}} \sum_{f \in (\basisB^X)^{H_{h,\phi}}} 
    \phi(f) \prod_{x \in \lmod{X}{H_{h,\phi}}} N(h(x),g(x)) M(g(x),f(x)),
  \]
  an expression that involves only finite sums and products.

  We now turn to $(NM)^{\tensor X} (\phi) \colon \basisD^X \to \F_2$, which is 
  given by
  \[
    \begin{split}
      h \mapsto & \phantom{{}=} \sum_{f \in \basisB^X} \phi(f) \prod_{x \in X} 
      (NM)(h(x),f(x))\\
      & = \sum_{f \in \basisB^X} \phi(f) \prod_{x \in X} \left(\sum_{c \in 
          \basisC} N(h(x),c)M(c,f(x))\right).
    \end{split}
  \]
  By the Proposition of \S~\ref{ssec:vs-func-mcc-1}, the above is equal to the 
  finite sum
  \[
    \sum_{f \in (\basisB^X)^{G_{h,\phi}}} \phi(f) \prod_{x \in X} \left(\sum_{c 
        \in \basisC} N(h(x),c)M(c,f(x))\right),
  \]
  where $G_{h,\phi}$ is the same group as in the previous paragraph. Consider 
  the group $H_{h,\phi}$ defined in the previous paragraph, which is an open 
  subgroup of $G_{h,\phi}$ and also fixes $h \colon X \to D$. By the 
  Proposition of \S~\ref{ssec:vs-F2-inf-prods-sums}, the index set 
  $(\basisB^X)^{G_{h,\phi}}$ in the sum above can be replaced by the index set 
  $(\basisB^X)^{H_{h,\phi}}$. Also, the infinite index set $X$ in the product 
  can be replaced by the finite index set $\lmod{X}{H_{h,\phi}}$. Therefore, 
  the expression above is equal to
  \[
    \sum_{f \in (\basisB^X)^{H_{h,\phi}}} \phi(f) \prod_{x \in 
      \lmod{X}{H_{h,\phi}}} \left(\sum_{c \in \basisC} 
      N(h(x),c)M(c,f(x))\right),
  \]
  in which all index sets are finite. For the finite product of finite sums in 
  the expression, we have the distributive property\footnote{The more general 
    fact is that, for finite sets $I$ and $J$,
    \[
      \prod_{i \in I} \left(\sum_{j \in J} a(i,j)\right) = \sum_{f \colon I \to 
        J} \left(\prod_{i \in I} a(i,f(i))\right).
    \]}
  \[
    \prod_{x \in \lmod{X}{H_{h,\phi}}} \left(\sum_{c \in \basisC} 
      N(h(x),c)M(c,f(x))\right)
    =
    \sum_{g \in \basisC^{\lmod{X}{H_{h,\phi}}}} \left(\prod_{x \in 
        \lmod{X}{H_{h,\phi}}} N(h(x),g(x)) M(g(x),f(x))\right).
  \]
  Thus, the value of $(NM)^{\tensor X} (\phi)$ on $h$ is equal to
  \[
    \begin{split}
      & \phantom{{}=} \sum_{f \in (\basisB^X)^{H_{h,\phi}}} \left(\sum_{g \in 
          \basisC^{\lmod{X}{H_{h,\phi}}}} \phi(f) \prod_{x \in 
          \lmod{X}{H_{h,\phi}}} N(h(x),g(x)) M(g(x),f(x))\right)\\
      & = \sum_{f \in (\basisB^X)^{H_{h,\phi}}} \left(\sum_{g \in 
          (\basisC^X)^{H_{h,\phi}}} \phi(f) \prod_{x \in \lmod{X}{H_{h,\phi}}} 
        N(h(x),g(x)) M(g(x),f(x))\right)\\
      & = \sum_{g \in (\basisC^X)^{H_{h,\phi}}} \left(\sum_{f \in 
          (\basisB^X)^{H_{h,\phi}}} \phi(f) \prod_{x \in \lmod{X}{H_{h,\phi}}} 
        N(h(x),g(x)) M(g(x),f(x))\right),
    \end{split}
  \]
  matching that of the value of $N^{\tensor X} (M^{\tensor X} (\phi))$ on $h$.
\end{proof}

\subsection{Staircase picture}
\label{ssec:vs-staircase}
The formula \eqref{eq:fun-BX-R-cc} defining $\Fun(\basisB^X,\F_2)^{\cc}$ as a 
union of $\colgrp$-invariant subspaces can also be written as a direct 
limit:\footnote{We write $\colgrp \supset \germ$ instead of $\grpG \supset 
  \germ$ in this subsection to be consistent with 
  \S~\ref{ssec:intro-staircase}.}
\begin{equation}
  \label{eq:mag-dir-lim}
  \Fun(\basisB^X,\F_2)^{\cc}:=\bigcup_{\colgrp \supset \germ} 
  \Fun(\basisB^X,\F_2)^{\colgrp} = \varinjlim_{\colgrp\supset \germ} 
  \Fun(\basisB^X,\F_2)^{\colgrp}.
\end{equation}
The direct limit, which may be taken in the category of sets or in the category 
of $\F_2$--vector spaces, is indexed by the opposite of the poset of groups 
whose germ is $\germ$.

\begin{proposition*}
  Each term $\Fun(\basisB^X,\F_2)^{\colgrp}$ in the direct limit 
  \eqref{eq:mag-dir-lim} is itself an inverse limit of finite-dimensional 
  $\F_2$--vector spaces: when $\rowgrp$ runs through the \emph{normal} open 
  subgroups of $\colgrp$, the restriction maps \begin{equation}
    \label{eq:vert}
    \Fun(\basisB^X,\F_2)^{\colgrp} \to \Fun(\basisB^{\rowgrp\backslash 
      X},\F_2)^{\colgrp}
  \end{equation}
  assemble to an isomorphism
  \begin{equation}
    \label{eq:profin-subsets}
    \Fun(\basisB^X,\F_2)^{\colgrp} \cong \varprojlim_{\rowgrp \onsubgrp 
      \colgrp} \left(
      \Fun(\basisB^{\rowgrp \backslash X},\F_2)^{\colgrp}\right).
  \end{equation}
\end{proposition*}

\begin{proof}
  By definition, $\basisB^X$ is the union over open subgroups $\colgrp \supset 
  \germ$ of the finite sets $\basisB^{\colgrp\backslash X}$. We may write this 
  union as a direct limit $\varinjlim_{\colgrp\supset \germ} \basisB^{\colgrp 
    \backslash X}$, where $\colgrp$ runs through the opposite of the poset of 
  groups whose germ is $\germ$. For a fixed $\colgrp \supset \germ$, the 
  \emph{normal} open subgroups $\rowgrp \onsubgrp \colgrp$ are cofinal in this 
  poset, and so we also have
  \[
    \basisB^X = \varinjlim_{\rowgrp \onsubgrp \colgrp} \basisB^{\rowgrp 
      \backslash X}.
  \]
  Now \eqref{eq:profin-subsets} is a consequence of the fact that a 
  representable functor like $\Fun(-,\F_2)$ converts colimits into limits, and 
  that taking $\colgrp$-invariants (for any group $\colgrp$) preserves limits:
  \[
    \begin{split}
      \Fun(\basisB^X,\F_2)^{\colgrp} & = \Fun \left(\varinjlim_{\rowgrp} 
        \basisB^{\rowgrp \backslash X},\F_2 \right)^{\colgrp} \\
      & = \left( \varprojlim_{\rowgrp} \Fun(\basisB^{\rowgrp \backslash 
          X},\F_2)\right)^{\colgrp} \\
      & = \varprojlim_{\rowgrp} \left(\Fun(\basisB^{\rowgrp \backslash 
          X},\F_2)^{\colgrp}\right).
      \qedhere
    \end{split}
  \]
\end{proof}

This subspace of $\Fun(\basisB^X,\F_2)^{\cc}$ is what we called in the 
introduction $F'_H$:
\begin{equation*}
  F'_H:=  \Fun(\basisB^X,\F_2)^H.
\end{equation*}
We have suppressed the dependence on $\basisB$ (and $X$, $\germ$, \dots) from 
the notation, but $\basisB \mapsto F'_H$ is a functor of $\basisB \in 
\Mat^{\F_2}$. The inclusion into $\Fun(\basisB^X,\F_2)^{\cc} =: 
\bigprofintensor{X} \vsV$ as well as the surjection
\begin{equation}
  \label{eq:F-prime-surj}
  F'_H \to \Fun(B^{H\backslash X},\F_2) \cong \bigtensor_{H\backslash X} \vsV
\end{equation}
are natural transformations. We denote the kernel of \eqref{eq:F-prime-surj} by 
$F''_H$.

Together, \eqref{eq:mag-dir-lim} and \eqref{eq:profin-subsets} give a 
description of $\Fun(\basisB^X,\F_2)^\cc$ as an inverse limit followed by a 
direct limit:
\begin{equation}
  \label{eq:ind-pro}
  \Fun(\basisB^X,\F_2)^\cc = \varinjlim_{\colgrp \supset \germ} 
  \varprojlim_{\rowgrp \onsubgrp \colgrp} \Fun(\basisB^{\rowgrp\backslash 
    X},\F_2)^{\colgrp}.
\end{equation}
We call \eqref{eq:ind-pro} the ``staircase picture'' of the magnetized tensor 
product, for the following reason. If $\colgrp^0 \supset \germ$ is second 
countable and of $2$-power order, then it is possible to choose a decreasing 
filtration of $\colgrp^0$ by open normal subgroups
\[
  \colgrp^0 \onsupgrp \colgrp^1 \onsupgrp \dotsb,
\]
so that each $\colgrp^i$ is normal not only in $\colgrp^{i-1}$ but in the whole 
group $\colgrp^0$, and so that $\bigcap \colgrp^i = \{1\}$. For short, write 
$\ast_{ij} = \Fun(\basisB^{\colgrp^j\backslash X},\F_2)^{\colgrp^i}$. The 
diagram of surjections \eqref{eq:vert} and inclusions \[
  \Fun(\basisB^{\colgrp^j\backslash X},\F_2)^{\colgrp^i} \to  
  \Fun(\basisB^{\colgrp^j\backslash X},\F_2)^{\colgrp^{i+1}}
\]
looks like
\[
  \xymatrix{
    \vdots  \ar[d] & \vdots  \ar[d] & \vdots \ar[d]  & \quad \iddots \\
    \ast_{02} \ar[d] \ar[r] & \ast_{12} \ar[d] \ar[r] & \ast_{22} \\
    \ast_{01} \ar[d] \ar[r] & \ast_{11} \\
    \ast_{00}
  }
\]
where every horizontal map is an inclusion and every vertical map is a 
surjection.

\subsection{Sectors of \texorpdfstring{$\bigprofintensor{X} \vsV$}{the 
    profinite tensor product}}
\label{ssec:vs-sector}
A subset $J \subset \basisB^X$ determines a direct summand of the space of all 
functions $\Fun(\basisB^X,\F_2)$, namely $\Fun(J,\F_2)$. The inclusion and 
projection maps
\[
  \Fun(J,\F_2) \to \Fun(\basisB^X,\F_2) \to \Fun(J,\F_2)
\]
are, respectively, extension-by-zero and restriction. 

If there is a $G \supset \germ$ for which $J \subset \basisB^X$ is $G$-stable, 
then $\Fun(J,\F_2)$ is a $G$-stable summand of $\Fun(\basisB^X,\F_2)$, and 
because of this
\[
  \Fun(J,\F_2)^{\cc} = \Fun(J,\F_2) \cap \Fun(\basisB^X,\F_2)^{\cc} = 
  \bigcup_{G \supset \germ} \Fun(J,\F_2)^{G}
\]
is a summand of $\Fun(\basisB^X,\F_2)^{\cc}$.

Now suppose $B = \amalg_{a \in A} B_a$, where $A$ is a finite index set. One 
may think of $B_a$ as the fiber of a map $B \twoheadrightarrow A$. Let $A^X$ 
(following the notation $B^X$) denote the set of continuous functions $X \to 
A$, let $\sectorA \subset A^X$ be a subset of $A^X$, and let $J_{\sectorA} = 
\{f \colon X \to B \, \vert \, \text{the composite } X \xrightarrow{f} B 
\twoheadrightarrow A \text{ is in } \sectorA\}$. If there is a $G \supset 
\germ$ for which $\sectorA \subset A^X$ is $G$-stable, then $J_{\sectorA} 
\subset B^X$ is stable for the same $G$.  In that case
\[
  \Fun(J_{\sectorA},\F_2)^{\cc}
\]
is a summand of $\Fun(B^X,\F_2)^{\cc}$ that we will call the 
\emph{$\sectorA$-sector} of $\Fun(B^X,\F_2)^{\cc}$, or equivalently of 
$\bigprofintensor{X} \vsV$.

The $\sectorA$-sector of $\bigprofintensor{X} \vsV$ is the magnetized and 
conditionally convergent analog of the following standard situation. Suppose 
that we have a direct sum decomposition $\vsV = \bigoplus_{a \in A} \vsV_a$, 
where $A$ is a finite set. If $X$ is also a finite set, then this 
$A$-decomposition of $\vsV$ induces a decomposition of the tensor power of $X$ 
copies of $\vsV$ into pieces that are indexed by functions $X \to A$. Since the 
set $A^X$ of such functions is finite, we can understand this decomposition 
both as a direct sum and as a direct product:
\[
  \bigtensor_X \vsV = \bigoplus_{\alpha \in A^X} \left(\bigtensor_{x \in X} 
    \vsV_{\alpha(x)}\right) = \prod_{\alpha \in A^X} \left(\bigtensor_{x \in X} 
    \vsV_{\alpha(x)}\right).
\]
The summand $\bigtensor_{x \in X} \vsV_{\alpha(x)}$ is spanned by those pure 
tensors $\bigtensor_{x \in X} v_x$ with $v_x$ in $\vsV_{\alpha(x)}$ for every 
$x$.
We will call this summand the \emph{$\alpha$-sector} of $\bigtensor_X \vsV$.  
More generally, if $\sectorA \subset A^X$, the direct sum (and equivalently the 
direct product) of those $\alpha$-sectors with $\alpha \in \sectorA$, \[
  \left(\bigtensor_{X} \vsV
  \right)_{\sectorA} = \bigoplus_{\alpha \in \sectorA} \left(\bigtensor_{x \in 
      X} \vsV_{\alpha(x)}\right),
\]
is the $\sectorA$-sector of $\bigtensor_{X} \vsV$.
If $X$ is a profinite set with a germ action $\germ$ of pro-$2$-power order, 
the subset $\sectorA \subset A^X$ determines an analogous summand of 
$\bigprofintensor{X} \vsV$ as long as $\sectorA$ is $G$-stable for some $G 
\supset \germ$.

\section{Tensor powers of bimodules}
\label{sec:bimod}

\subsection{Solenoidal structures}
\label{ssec:bimod-solenoid}
A \emph{multicyclic structure} $X_S$ is a set $X$ together with a bijection $S 
\colon X \to X$. The name comes from the observation that if $X$ is a finite 
set, each orbit of $S$ on $X$ has a cyclic order in an evident way: $x$ 
precedes $Sx$ in the order. We go on calling it a multicyclic structure even 
when some orbits of $S$ are infinite, and even when every orbit of $S$ is 
infinite.

We say that a multicyclic structure is compatible with a topology on $X$ if $S$ 
is a homeomorphism. If $X$ is profinite and $\germ$ is a germ action on $X$ 
(\S~\ref{ssec:vs-profin-sets}), we say that $S$ is compatible with $\germ$ if 
there is at least one $G \supset \germ$ such that $S$ commutes with the action 
of every $g \in G$:
\[
  gS(x) = S(gx) \text{ for all $x \in X$ and $g \in G$}.
\]
In that case we call $S$ a \emph{solenoidal} structure, for the reason given in 
\S~\ref{ssec:intro-bimod-fin}.

\subsection{Basis and magnetized pure tensors}
\label{ssec:bimod-mag}
Let $R$ be a commutative ring, and let $\idemring$ be a product of finitely 
many copies (say, $n$ copies) of $R$, regarded as an associative $R$-algebra.  
Write \begin{equation}
  \label{eq:prim-idem}
  \idem_1,\idem_2,\dotsc,\idem_n \in \idemring
\end{equation}
for the primitive idempotents of $\idemring$. Let $\bimodV$ be an 
$(\idemring,\idemring)$-bimodule that is free of finite rank over $R$. In fact, 
suppose that $\bimodV$ has the stronger property that each of the subsets 
$\idem_k \bimodV \idem_\ell \subset \bimodV$ is free of finite rank over $R$.

Soon we will require that $R = \F_2$, and that $\germ$ is the germ of an action 
by a pro-$2$-group, but this hypothesis plays no role until 
\S~\ref{ssec:bimod-func-mcc}.

Choose an $R$ basis for each $\iota_k \bimodV \iota_{\ell}$, and denote it by
\begin{equation*}
  \label{eq:B-edges}
  \graphB:=\amalg_{k,\ell} \basisB_{k,\ell}.
\end{equation*}
The elements of $\amalg_{k,\ell} \basisB_{k,\ell}$ can be regarded as the edges 
of a directed graph whose vertices are the primitive idempotents 
\eqref{eq:prim-idem}. Each element $b \in \graphB_{k,\ell}$ is an arrow in the 
directed graph from vertex $\iota_k$ to vertex $\iota_{\ell}$. If $e$ is an 
edge, write $s(e)$ for the vertex it points from and $t(e)$ for the vertex it 
points to, i.e.
\[
  e \in \graphB_{k,\ell} \implies s(e) =\iota_k \text{ and }t(e) = 
  \iota_{\ell}.
\]
We sometimes denote the graph itself by $(\graphB,s,t)$.

Let $X$ and $\germ$ be as in \S~\ref{ssec:vs-profin-sets} and let $S \colon X 
\to X$ be a compatible multicyclic structure in the sense of 
\S~\ref{ssec:bimod-solenoid}. Let $\graphB^{X_S}$ denote the set of continuous 
functions $f \colon X \to \graphB$ with the following property:
\begin{equation}
  \label{eq:S-next-elem}
  \text{For all $x \in X$, $s(f(Sx)) = t(f(x))$}.
\end{equation}
The topology on the finite set $\graphB$ is the discrete topology. A continuous 
function must factor through one of the finite quotients $G \backslash X$, 
where $G \supset \germ$. The condition \eqref{eq:S-next-elem} says that each 
$S$-orbit on $G \backslash X$ is mapped to an oriented cycle in the directed 
graph.

\subsection{Conditionally convergent tensors in the solenoidal sector}
\label{ssec:bimod-cc}
The elements of $\graphB^{X_S}$ are the ``pure tensors'' in our tensor power 
$\bigprofintensor{X_S} \bimodV$, with tensor factors chosen from the basis 
$\graphB$. A more general tensor is, in a unique way, an $R$-linear combination 
of pure tensors. As in \S~\ref{ssec:vs-cc}, the most general kind of linear 
combination (with possibly infinitely many terms) can be modeled as a function 
in
\(
  \Fun(\graphB^{X_S},R).
\)
and we identify a subgroup of conditionally convergent tensors by the formula 
\begin{equation}
  \label{eq:fun-BXS-R-cc}
  \Fun(\graphB^{X_S},R)^\cc := \bigcup_{G \supset \germ} 
  \Fun(\graphB^{X_S},R)^{G}.
\end{equation}

In fact, $\Fun(\graphB^{X_S},R)^{\cc}$ is a sector of $\Fun(\graphB^X,R)^\cc$ 
in the sense of \S~\ref{ssec:vs-sector}. Put $A = 
\{\iota_1,\dotsc,\iota_n\}^{\times 2}$, the set of ordered pairs of primitive 
idempotents in \eqref{eq:prim-idem}. The solenoidal structure $S$ on $X$ 
determines a subset $\sectorA_S$ of $A^X$.  Specifically,
\(
  (a_L,a_R) \colon X \to \{\iota_1,\dotsc,\iota_n\}^{\times 2}
\)
belongs to $\sectorA_S$ if
\[
  \text{for all } x \in X, a_L(Sx) = a_R(x).
\]
The set $\sectorA_S$ is $G$-stable as a subset of $A^X$ for any $G \supset 
\germ$, and $\Fun(\graphB^{X_S},R)^{\cc}$ is the $\sectorA_S$-sector of 
$\Fun(\graphB^X,R)^{\cc}$. The idempotent projection
\[
  e_{S,\graphB} \colon \Fun(\graphB^X,R)^{\cc} \to \Fun(\graphB^X,R)^{\cc}
\]
is given by the formula
\begin{equation}
  \label{eq:eS}
  e_{S,\graphB}(\phi)(f) =
  \begin{cases}
    \phi(f) & \text{if } f \in \graphB^{X_S}, \\
    0 & \text{otherwise}.
  \end{cases}
\end{equation}

\subsection{Change of basis and other functorialities for bimodules}
\label{ssec:bimod-change-of-basis}
As in \S~\ref{ssec:vs-change-of-basis}, it is not obvious that the magnetized 
and conditionally convergent tensor product $\Fun(\graphB^{X_S},R)^{\cc}$ that 
we've defined is canonically attached to the bimodule $\bimodV$. In other words 
it is not obvious that, if for each  $k$ and $\ell$ we give another basis 
$\graphC_{k,\ell}$  for $\iota_k \bimodV \iota_{\ell}$, that the 
change-of-basis matrices (from $\graphB_{k,\ell}$ to $\graphC_{k,\ell}$) induce 
an isomorphism
\begin{equation*}
  \Fun(\graphB^{X_S},R)^{\cc} \cong \Fun(\graphC^{X_S},R)^{\cc}.
\end{equation*}
To analyze this, we will in the next subsection define a bimodule analog 
$\Mat^{(\idemring,\idemring)}$ of the category $\Mat^R$ of finite sets and 
matrices, reviewed in \S~\ref{ssec:vs-cat-fin-matrices}.

\subsection{Matrices for 
  \texorpdfstring{$(\idemring,\idemring)$-modules}{(I,I)-modules}}
\label{ssec:bimod-cat-fin-matrices}
We define a category $\Mat^{(\idemring,\idemring)}$:
\begin{itemize}
  \item The objects are tuples $(\graphB,s,t)$, where $\graphB$ is a finite set 
    and $s$ and $t$ are two functions $\graphB \rightrightarrows 
    \{\iota_1,\dotsc,\iota_n\}$ valued in the set of primitive idempotents of 
    $\idemring$. We view each $(\graphB,s,t)$ as a directed graph, possibly 
    with loops and multiple edges, whose vertex set is 
    $\{\iota_1,\dotsc,\iota_n\}$.
  \item A morphism
    \begin{equation}
      \label{eq:graphB-M-graphC}
      \left[(\graphB,s,t) \xrightarrow{M} (\graphC,s,t)\right]
    \end{equation}
    is a function $M \colon \graphC \times \graphB \to R$ with the property 
    that
    \begin{equation}
      \label{eq:contrapositive-equals}
      \text{$M(c,b) = 0$ unless $s(c) = s(b)$ and $t(c) = t(b)$.}
    \end{equation}
  \item The composition law is the same formula as \eqref{eq:NM-d-b},
    \begin{equation*}
      (NM)(d,b) = \sum_{c \in \graphC} N(d,c)M(c,b).
    \end{equation*}
    The contrapositive of the transitivity of $=$ shows that 
    \eqref{eq:contrapositive-equals} holds for $NM$ whenever it holds for both 
    $N$ and $M$.
\end{itemize}

Given $f \colon \graphB \to R$ and a pair of elements $\lambda \iota_k$ and 
$\rho \iota_\ell$ in $\idemring$, define a new function $(\lambda \iota_k)\cdot 
f \cdot (\rho \iota_\ell) \colon \graphB \to R$ by the formula
\begin{equation}
  \label{eq:bimod-Fun-B-R}
  (\lambda \iota_k)\cdot f\cdot (\rho \iota_{\ell}) (b) =
  \begin{cases}
    \lambda \rho f(b) & \text{if } s(b) = \iota_k \text{ and } t(b)= 
    \iota_\ell,\\
    0 & \text{otherwise}.
  \end{cases}
\end{equation}
This formula defines a $(\idemring,\idemring)$-bimodule structure on 
$\Fun(\graphB,R)$ that is free as an $R$-module. In fact, each idempotent 
subspace $\iota_k\Fun(\graphB,R)\iota_{\ell}$ is free as an $R$-module.

The assignment $\graphB \mapsto \Fun(\graphB,R)$ is covariantly functorial for 
matrices \eqref{eq:graphB-M-graphC} by the same rule
as in \S~\ref{ssec:vs-cat-fin-matrices}:
\begin{gather*}
  f \in \Fun(\graphB,R), \quad f \mapsto Mf \in \Fun(\graphC,R);\\
  (Mf)(c) = \sum_{b \in \graphB} M(c,b)f(b).
\end{gather*}
This functor is an equivalence between $\Mat^{(\idemring,\idemring)}$ and the 
full subcategory of $(\idemring,\idemring)$-bimodules $\bimodV$ whose 
idempotent summands $\iota_k \bimodV \iota_{\ell}$ are all free $R$-modules of 
finite rank.

\subsection{Finite multicyclic tensor powers}
\label{ssec:bimod-func-fin-tensor}
Let $(\graphB,s,t)$ be an object of $\Mat^{(\idemring,\idemring)}$. Let $X$ be 
a \emph{finite} set and let $S \colon X \to X$ be a bijection, i.e.\ a 
multicyclic structure on $X$. As in \S~\ref{ssec:bimod-mag}, let 
$\graphB^{X_S}$ denote the set of functions from $X$ to $\graphB$ with
\[
  s(f(Sx))= t(f(x)) \qquad \left(\begin{array}{c}
      \text{equivalently, with $s(f(x))= t(f(S^{-1}x))$,}\\
      \text{or with $s(f(x)) = t((Sf)(x))$}
    \end{array}
  \right)
\]
for all $x \in X$. The assignment
\begin{equation}
  \label{eq:Fun-BXS-R}
  (\graphB,s,t) \mapsto \Fun(\graphB^{X_S},R)
\end{equation}
is functorial for the kind of matrices of \S~\ref{ssec:bimod-change-of-basis}, 
in a sense we now explain. Since $X$ is finite, $\graphB^{X_S}$ is also finite 
and \eqref{eq:Fun-BXS-R} takes values in the category of free $R$-modules of 
finite rank. But it does not take values in the category of 
$(\idemring,\idemring)$-bimodules.

Given $(\graphB,s,t)\xrightarrow{M}(\graphC,s,t)$, define a map $M^{\tensor 
  X_S} \colon \Fun(\graphB^{X_S},R) \to \Fun(\graphC^{X_S},R)$ by the rule
\[
  (M^{\tensor X_S} \phi)(g) = \sum_{f \in \graphB^{X_S}} \phi(f) \prod_{x \in 
    X} M(g(x),f(x)).
\]

As we discussed in \S~\ref{ssec:intro-bimod-fin}, there are two reasonably 
natural ways to take the tensor power of $n$ copies of the 
$(\idemring,\idemring)$-bimodule $\Fun(\graphB,R)$:
\begin{itemize}
  \item ``In a row'', i.e.\
    \[
      \Fun(\graphB,R) \tensor_{\idemring} \Fun(\graphB,R) \tensor_{\idemring} 
      \dotsb \tensor_{\idemring} \Fun(\graphB,R),
    \]
    which returns a new $(\idemring,\idemring)$-bimodule.  \item ``In a 
    cycle'', a natural formula for which is
    \[
      \HH_0\Big(\idemring,\Fun(\graphB,R) \tensor_{\idemring} \Fun(\graphB,R) 
      \tensor_{\idemring} \dotsb \tensor_{\idemring} \Fun(\graphB,R)\Big).
    \]
\end{itemize}
If $S$ is transitive and $X$ has $n$ elements (that we can denote by $x_0$, 
$x_1:=S(x_0)$, \dots, $x_{n-1}:=S^{n-1}(x_0)$), then $\Fun(\graphB^{X_S},R)$ is 
naturally isomorphic to the in-a-cycle version: we define a map
\begin{equation}
  \label{eq:Fun-BXS-R-HH-trans}
  \Fun(\graphB^{X_S},R) \to \HH_0\left(\idemring,\Fun(\graphB,R) 
    \tensor_{\idemring} \Fun(\graphB,R) \tensor_{\idemring} \dotsb 
    \tensor_{\idemring} \Fun(\graphB,R)\right)
\end{equation}
by the formula
\begin{equation*}
  \phi \mapsto \sum_f \phi(f) \cdot [\delta_{f(x_0)} \tensor \delta_{f(x_1)} 
  \tensor \dotsb \tensor \delta_{f(x_{n-1})}],
\end{equation*}
where $\delta_b \colon \graphB \to R$ takes the value $1$ at $b$ and $0$ at 
every other element of $\graphB$, and, for a general 
$(\idemring,\idemring)$-bimodule $M$, $[m]$ denotes the image of $m \in M$ 
under the natural projection $M \to \HH_0(\idemring,M)$.

The map \eqref{eq:Fun-BXS-R-HH-trans} is a natural isomorphism making the 
following triangle of functors commute:
\[
  \xymatrix{
    \Mat^{(\idemring,\idemring)} \ar[rrrr]^{\graphB \mapsto \Fun(\graphB,R) 
      \quad \eqref{eq:bimod-Fun-B-R}} \ar[drr]_{\graphB \mapsto 
      \Fun(\graphB^{X_S},R); M \mapsto M^{\tensor X,S} \qquad \qquad} & & & & 
    \text{$(\idemring,\idemring)$-bimodules} \ar[dll]^{\qquad \quad \bimodV 
      \mapsto \HH_0(\idemring,\bimodV \tensor \dotsb \tensor \bimodV)} \\
    & & \Mod(R)
  }
\]

What if $S$ is not transitive, and acts on $X$ with several orbits 
$X_1,X_2,\dotsc,X_N$? Then each orbit $X_i$ determines both an in-a-row
tensor power and an in-a-cycle tensor power.
The in-a-row tensor power is $\Fun(\graphB,R) \tensor_{\idemring} \dotsb 
\tensor_{\idemring} \Fun(\graphB,R)$, with tensor factors indexed by elements 
of $X_i$, which we write as \[
  x_0^i, x_1^i := Sx_0^i,\dotsc,x_{n_i-1}^i := S^{n_i-1} x_0^i.
\] For short, let us call this in-a-row tensor power $M_i$. The  in-a-cycle 
tensor power is $\HH_0(\idemring,M_i)$.
Then there is a natural isomorphism
\begin{equation}
  \label{eq:Fun-BXS-R-HH-not-trans}
  \Fun(\graphB^{X_S},R) \cong \bigtensor_{i=1}^N \HH_0(\idemring,M_i).
\end{equation}
To be explicit, consider a $\phi \in \Fun(\graphB^{X_S},R)$ that is equal to a 
product of functions:
\begin{equation}
  \label{eq:sep-of-vars}
  \phi(f)=\phi_1(f\vert_{X_1})\phi_2(f\vert_{X_2})\dotsm\phi_N(f\vert_{X_N}).
\end{equation}
The map $(\phi_1,\phi_2,\dotsc,\phi_n) \mapsto \phi$ defined by the rule 
\eqref{eq:sep-of-vars} is multilinear and extends to an isomorphism 
$\bigtensor_i \Fun(\graphB^{X_{i,S}},R) \stackrel{\sim}{\to} 
\Fun(\graphB^{X_S},R)$.) Then \eqref{eq:Fun-BXS-R-HH-not-trans} sends $\phi$ to
\[
  \bigtensor_{i = 1}^N  \left(\sum_f \phi_i(f) \cdot \left[\delta_{f(x_{0}^i)} 
      \tensor \delta_{f(x_{1}^i)} \tensor \dotsb \tensor 
      \delta_{f(x_{n_i-1}^i)}\right]\right).
\] 

\subsection{Functoriality}
\label{ssec:bimod-func-mcc}
For the same reason as in \S~\ref{ssec:vs-F2-inf-prods-sums}, we now restrict 
to the case where $R = \F_2$. Let $(\graphB,s,t)$ and $(\graphC,s,t)$ be two 
objects of the category $\Mat^{(\idemring,\idemring)}$. A morphism between them 
is a matrix
\[
  M\colon C \times B \to \F_2
\]
whose $(c,b)$ entry is zero unless $s(c) = s(b)$ and $t(c) = t(b)$. Let us 
explain how such an $M$ induces a map
\[
  M^{\tensor X_S} \colon \Fun(\graphB^{X_S},\F_2)^{\cc} \to 
  \Fun(\graphC^{X_S},\F_2)^{\cc}.
\]
In fact, given \emph{any} matrix $M \colon C \times B \to \F_2$, we have a 
reasonably natural formula for a map 
\(
  \Fun(\graphB^{X_S},\F_2)^{\cc} \to \Fun(\graphC^{X_S},\F_2),
\)
using the projectors in \eqref{eq:eS}.
If we make the identifications
\[
  \Fun(\graphB^{X_S},\F_2)^{\cc} = e_{\graphB,S}\Fun(\graphB^X,\F_2)^{\cc}
  \qquad \text{and} \qquad
  \Fun(\graphC^{X_S},\F_2)^{\cc} = e_{\graphC,S}\Fun(\graphC^X,\F_2)^{\cc},
\]
the formula is
\[
  M^{\tensor X_S}:=e_{\graphC,S} \circ M^{\tensor X}\vert_{e_{\graphB,S} 
    \Fun(\graphB^X,\F_2)^{\cc}}.
\]
where $M^{\tensor X}$ is as in \S~\ref{ssec:vs-func-mcc-1}. More explicitly,
\[
  (M^{\tensor X_S} \phi)(g) =
  \begin{cases}
    \sum_{f \in B^{X_S}} \phi(f) \prod_{x \in X} M(g(x),f(x)) & \text{if $g \in 
      C^{X_S}$},\\
    0 & \text{otherwise}.
  \end{cases}
\]
For arbitrary matrices $M$ and $N$ it is not the case that
\[
  (N \circ M)^{\tensor X_S} = N^{\tensor X_S} \circ M^{\tensor X_S}.
\]

\subsection*{Proposition}
Let $M \colon \graphC \times \graphB \to \F_2$ and $N \colon \graphD \times 
\graphC \to \F_2$ be morphisms in $\Mat^{(\idemring,\idemring)}$. Then
\[
  (N \circ M)^{\tensor X_S} = N^{\tensor X_S} \circ M^{\tensor X_S}.
\]
In other words, $\graphB \mapsto \Fun(\graphB^{X_S},\F_2)^{\cc}$ is a functor 
from $\Mat^{(\idemring,\idemring)}$ to the category of $\F_2$--vector spaces.

\begin{proof}
  A sufficient condition for $(N M)^{\tensor X_S} = N^{\tensor X_S} \circ 
  M^{\tensor X_S}$ is that, for all $(h, f) \in \graphD^{X_S} \times 
  \graphB^{X_S}$,
  \[
    \sum_{g \in \graphC^X - \graphC^{X_S}} \phi(f) \prod_{x \in X}  
    N(h(x),g(x))M(g(x),f(x)) = 0.
  \]
  The proof is almost identical to the proof of the Proposition of 
  \S~\ref{ssec:vs-func-mcc-2}:
  if $h \in \graphD^{X_S}$ or even if $h \in \graphD^X$, there is a $G_{h,\phi} 
  \supset \germ$ for which the assignment
  \[
    g \mapsto (M^{\tensor X_S}\phi(g)) \prod_{x \in X} N(h(x),g(x))
  \]
  is $G_{h,\phi}$-invariant, and for each $g \in (\graphC^{X_S})^{G_{h,\phi}}$, 
  there is a $G'_{g,\phi} \supset \germ$ for which the assignment
  \[
    f \mapsto \phi(f) \prod_{x \in X} M(g(x),f(x))
  \]
  is $G'_{g,\phi}$-invariant. Set
  \(
    G'_{h,\phi} := \bigcap_{g \in (\graphC^{X_S})^{G_{h,\phi}}} G'_{g,\phi}.
  \)
  There is also a $G''_h \supset \germ$ such that $h \colon X \to D$ is 
  constant on the orbits of $G''_h$. Setting
  \(
    H_{h,\phi} := G_{h,\phi} \cap G'_{h,\phi} \cap G''_h,
  \)
  the value of $N^{\tensor X_S} (M^{\tensor X_S} (\phi))$ on $h$ is
  \[
    \begin{split}
      & \phantom{{}=} \sum_{g \in (\graphC^{X_S})^{H_{h,\phi}}}
      \left(\sum_{f \in (\graphB^{X_S})^{H_{h,\phi}}}
        \phi(f) \prod_{x \in \lmod{X}{H_{h,\phi}}} 
        N(h(x),g(x))M(g(x),f(x))\right)\\
      &=
      \sum_{f \in (\graphB^{X_S})^{H_{h,\phi}}}
      \left(\sum_{g \in (\graphC^{X_S})^{H_{h,\phi}}}
        \phi(f) \prod_{x \in \lmod{X}{H_{h,\phi}}} 
        N(h(x),g(x))M(g(x),f(x))\right).
    \end{split}
  \]
  The inner sum is
  \begin{equation}
    \label{eq:N-of-M-exp}
    \sum_{g \in (\graphC^{X_S})^{H_{h,\phi}}}
    \phi(f) \prod_{x \in \lmod{X}{H_{h,\phi}}} N(h(x),g(x))M(g(x),f(x)).
  \end{equation}
  On the other hand, the value of $(NM)^{\tensor X_S} (\phi))$ on $h$ is
  \[
    \begin{split}
      & \phantom{{}=}
      \sum_{f \in (\graphB^{X_S})^{H_{h,\phi}}}
      \phi(f)
      \prod_{x \in \lmod{X}{H_{h,\phi}}}
      \left(\sum_{c \in C}
        N(h(x),g(x))M(g(x),f(x)) \right)\\
      & =
      \sum_{f \in (\graphB^{X_S})^{H_{h,\phi}}}
      \left(\sum_{g \in (\graphC^{X})^{H_{h,\phi}}}
        \phi(f)
        \prod_{x \in \lmod{X}{H_{h,\phi}}}
        N(h(x),g(x))M(g(x),f(x)) \right),
    \end{split}
  \]
  whose inner sum is
  \begin{equation}
    \label{eq:NM-exp}
    \begin{split}
      & \phantom{{}=} \sum_{g \in (\graphC^{X})^{H_{h,\phi}}} \phi(f) \prod_{x 
        \in \lmod{X}{H_{h,\phi}}} N(h(x),g(x))M(g(x),f(x))\\
      & = \sum_{g \in (\graphC^{X_S})^{H_{h,\phi}}} \phi(f) \prod_{x \in 
        \lmod{X}{H_{h,\phi}}} N(h(x),g(x))M(g(x),f(x))\\
      & \phantom{{}=} + \sum_{g \in (\graphC^{X} \setminus 
        \graphC^{X_S})^{H_{h,\phi}}} \phi(f) \prod_{x \in \lmod{X}{H_{h,\phi}}} 
      N(h(x),g(x))M(g(x),f(x)).
    \end{split}
  \end{equation}
  The inner sums \eqref{eq:N-of-M-exp} and \eqref{eq:NM-exp} are equal if
  \[
    \sum_{g \in (\graphC^X - \graphC^{X_S})^{H_{h,\phi}}} \phi(f) \prod_{x \in 
      \lmod{X}{H_{h,\phi}}}  N(h(x),g(x))M(g(x),f(x)) = 0.
  \]

  We now claim that, for each $g \in \graphC^X \setminus \graphC^{X_S}$, we in 
  fact have
  \[
    \prod_{x \in X}  N(h(x),g(x))M(g(x),f(x)) = 0
  \]
  for all $f \in \graphB^{X_S}$ and $h\in \graphD^X$ (sic).  Indeed, recall 
  that $g \in C^{X_S}$ when $s(g(Sx)) = t(g(x))$ for all $x$.  If $g \in C^X - 
  C^{X_S}$, then there is some $x$ for which $s(g(Sx)) \neq t(g(x))$. It 
  follows that for all $f \in B^{X_S}$ (i.e.\  for all $f$ with $s(f(Sx)) = 
  t(f(x))$ for all $x$), there is an $x$ such that
  \[
    s(g(x)) \neq s(f(x))
  \]
  and for such an $x$, we have $M(g(x),f(x)) = 0$.
\end{proof}

\subsection{Staircase picture}
\label{ssec:bimod-staircase}
The same constructions in \S~\ref{ssec:vs-staircase} make sense when $\graphB$ 
is a basis for an $(\idemring,\idemring)$-bimodule.
By definition \eqref{eq:fun-BXS-R-cc}, $\Fun(\graphB^{X_S},\F_2)^{\cc}$ is the 
union over open subgroups $\colgrp \supset \germ$ of the $\colgrp$-invariant 
subspaces $\Fun(\graphB^{X_S},\F_2)^{\colgrp}$. Each of these subspaces is the 
inverse limit over open normal subgroups $\rowgrp \onsubgrp \colgrp$ of the 
finite-dimensional spaces $\Fun(\graphB^{\rowgrp \backslash 
  X_S},\F_2)^{\colgrp}$, and so
\begin{equation*}
  \Fun(\graphB^{X_S},\F_2)^{\cc} \cong \varinjlim_{\colgrp \supset \germ} 
  \varprojlim_{\rowgrp \onsubgrp \colgrp} \Fun(\graphB^{\rowgrp \backslash 
    X_S},\F_2)^{\colgrp}.
\end{equation*}
For each $\colgrp \supset \germ$, define subgroups $F''_\colgrp \subset 
F'_\colgrp \subset \Fun(\graphB^{X_S},\F_2)^\cc$ by the formulas
\begin{equation*}
  F'_\colgrp := \Fun(\graphB^{X_S},\F_2)^\colgrp, \qquad F''_\colgrp 
  :=\operatorname{ker}\left(
    F'_\colgrp \to \Fun(\graphB^{\colgrp \backslash X_S},\F_2)^\colgrp
  \right).
\end{equation*}
These are functorial for maps in $\Mat^{(\idemring,\idemring)}$. When $\colgrp 
\backslash X$ has $n$ elements and $S$ acts transitively, the quotient $
F'_\colgrp/F''_\colgrp \cong \Fun(\graphB^{\colgrp \backslash 
  X_S},\F_2)^\colgrp
$ is naturally isomorphic to the Hochschild homology of the in-a-line tensor 
product of $n$ copies of $\bimodV$
\[
  F'_\colgrp/F''_\colgrp \cong \HH_0(\idemring,\bimodV \tensor_{\idemring} 
  \dotsb \tensor_{\idemring} \bimodV).
\]
When $S$ has several orbits on $\colgrp \backslash X$, it is isomorphic to a 
tensor product of several $\HH_0$'s, as in \eqref{eq:Fun-BXS-R-HH-not-trans}.

\section{An example from Heegaard Floer homology}
\label{sec:example}

In this section we will illustrate an example of 
$\bigprofintensor{X_S}{\bimodV}$ when $X$ is an affine copy of $\Z_2$, and 
$\germ$ is the germ of the action of $\Z_2$ on itself. The solenoidal structure 
$S \colon X \to X$ is $S(x) = x+1$.
The example is closely related to the knot Heegaard Floer homologies of the 
pair \[
  (S^3,4_1):=(S^3,\text{figure-eight knot})
\]
and of its $2^n$-fold branched cyclic covers, or equivalently to the sutured 
Heegaard Floer homologies of the exterior of the figure-eight knot and its 
$2^n$-fold cyclic covers. It suggests to us an interpretation of 
$\bigprofintensor{X_S} \bimodV$ as the sutured Heegaard Floer homology of a 
``pro-manifold'' of dimension $3$.

\subsection{A semisimple ring and bimodule}
\label{ssec:example-ring-bimod}
Throughout this section, $\idemring$ will be the semisimple ring that is the 
product of $4$ copies of the field $\F_2$. We label its primitive idempotents 
by subsets of $\{\zero,\one\}$, as in
\begin{equation}
  \label{four-vertices}
  \iota_{\varnothing}, \iota_{\{\zero\}}, \iota_{\{\one\}}, 
  \iota_{\{\zero,\one\}}.
\end{equation}

Let $\graphB$ be the directed graph whose vertices are \eqref{four-vertices} 
and whose seven edges
\begin{equation}
  \label{eq:seven-edges}
  t,u,v,w,x,y,z
\end{equation}
are as in the figure:
\[
  \begin{tikzpicture}[scale=0.6]
    \node (i0) at (-2,0) {$\iota_\varnothing$};
    \draw[->, out=192, in=168, looseness=15] (i0) to node[above left] {$t$} 
    (i0);
    \node (i1) at (2.5,0) {$\iota_{\{0\}}$};
    \node (i2) at (5.5,0) {$\iota_{\{1\}}$};
    \draw[->, out=15, in=165] (i1) to node[above] {$v$} (i2);
    \draw[->, out=195, in=345] (i2) to node[below] {$w$} (i1);
    \draw[->, out=190, in=170, looseness=15] (i1) to node[above left] {$u$} 
    (i1);
    \draw[->, out=10, in=40, looseness=10] (i2) to node[above right] {$x$} 
    (i2);
    \draw[->, out=350, in=320, looseness=10] (i2) to node[below right] {$y$} 
    (i2);
    \node (i3) at (10,0) {$\iota_{\{0,1\}}$};
    \draw[->, out=10, in=350, looseness=10] (i3) to node[above right] {$z$} 
    (i3);
  \end{tikzpicture}
\]
Let $\bimodV = \F_2\{t,u,v,w,x,y,z\}$ be the $\F_2$--vector space spanned by 
the set of edges in this graph, with its $(\idemring,\idemring)$-bimodule 
structure.

There is a Seifert surface of the figure-eight knot $\figeight$ with genus one 
and one boundary component (the knot itself). Like any torus with one boundary 
component, this Seifert surface can be presented as a disk with two $1$-handles 
attached to it. Later on, the set $\{\zero,\one\}$ indexes the $1$-handles in 
such a presentation. The graph comes in a more complicated way from a Heegaard 
decomposition of the bordered $3$-manifold obtained from the exterior of 
$\figeight$ by cutting along this Seifert surface.

\subsection{Monomial, polynomial, and series notation}
\label{ssec:example-notation}
The set of pure tensors in $\bigprofintensor{X_S} \vsV$ is $\graphB^{X_S}$.  
Let us describe $\graphB^{X_S}$ in more detail. In \S~\ref{ssec:bimod-mag}, it 
is defined as the set of magnetized functions from $X$ to the edges of 
$\graphB$ obeying \eqref{eq:S-next-elem}.
When $X = \Z_2$, a magnetized function is constant on cosets of $2^m \Z_2$ for 
some $m$. The solenoidal structure $S$ with $S(x) = x+1$ permutes the cosets 
cyclically, and so an element of $\graphB^{X_S}$ is determined by a directed 
cycle whose length is a power of $2$.
We can name such a function by giving a sequence of letters from $tuvwxyz$ 
\eqref{eq:seven-edges} whose length is a power of $2$. Let us give some 
examples and non-examples:
\begin{itemize}
  \item The single letter $u$, regarded as a sequence of length $2^0$.  
    Formally, this is the constant function $X \to \graphB$ that takes the 
    value $u$ on every element of $X$.
  \item The single letter $v$ does \emph{not} define a magnetized pure tensor, 
    because $v$ is not a loop.
  \item The directed cycle $vw$, which takes the value $v$ on even elements of 
    $X$ and $w$ on odd elements of $X$. Note that $vwvw$ indicates the same 
    function.
  \item The directed cycle $wv$, which takes the value $w$ on even elements of 
    $X$ and $v$ on odd elements of $X$. Note that $wv$ and $vw$ indicate 
    different functions.
  \item The directed cycle $vwuu$, which takes the value $v$ on elements of $X$ 
    that are $0$ mod $4$, $w$ on elements that are $1$ mod $4$, and $u$ on 
    elements that are either $2$ mod $4$ or $3$ mod $4$.
  \item The directed cycles $uuuvwuuu$, $uvwuuuuu$, and so on. To make these 
    easier to read, we can write them as though they were monomials in 
    noncommutative variables:
    \[
      uuuvwuuu = u^3 vw u^3, \qquad uvwuuuuu = u vw u^5,
    \]
    etc. However, beware that the analogy with monomials obscures the fact that 
    (for instance) $vw = vwvw$.
\end{itemize}
More formally, we have
\[
  \graphB^{X_S} = \bigcup_{G \supset \germ} \graphB^{(G\backslash X)_S} = 
  \bigcup_{m \in \Z_{\geq 0}} \graphB^{(2^m \Z_2 \backslash X)_S},
\]
and when we write a sequence of letters of length $2^m$, we are giving an 
element of $\graphB^{(2^m \Z_2 \backslash X)_S}$.
A general element of $\bigprofintensor{X_S} \vsV$ is a conditionally convergent 
sum of magnetized pure tensors. Let us give some examples and non-examples of 
the conditionally convergent property:
\begin{itemize}
  \item A magnetized pure tensor, viewed as a sum of one term, is conditionally 
    convergent. For example, $vw$ and $wv$ are invariant by $2\Z_2 \supset 
    \germ$, and $vwuu$ is invariant by $4\Z_2 \supset \germ$.
  \item A finite linear combination of magnetized pure tensors is conditionally 
    convergent. For example, $vw+wvuu$ is invariant by $4\Z_2$.
  \item A negative example: the sum $x + xy + xxxy + xxxxxxxy + \dotsb$ is 
    \emph{not} conditionally convergent. Formally, this is the function 
    $\graphB \to \F_2$ that takes the value $1$ on any pure tensor of the form 
    $x^{2^n - 1} y$ (which is magnetized) and $0$ on any other magnetized pure 
    tensor. In the shorthand of this section we can write it instead as
    \[
      \sum_{n = 0}^{\infty} x^{2^n-1} y.
    \]
    The $n$th term in this sum is invariant for the group $2^n \Z_2 \supset 
    \germ$, but the intersection of these groups is the trivial group. Since 
    the germ of the trivial group action is not $\germ$, the sum is not 
    conditionally convergent.
  \item The series
    \[
      \begin{split}
        & \phantom{{}+} (xy + yx) + (xxxy + xxyx + xyxx+ yxxx) \\
        & + (x^7 y + x^6yx + \dotsb + xyx^6 + yx^7) + \dotsb + \left(\sum_{i = 
            1}^{2^n} x^{2^n - i} y x^{i-1}\right) + \dotsb
      \end{split}
    \]
    is conditionally convergent. The $n$th term $(\sum_{i = 1}^{2^n} x^{2^n - 
      i} y x^{i-1})$ is invariant by the group $\Z_2$, and so the whole series 
    is also invariant by $\Z_2$.
\end{itemize}

\subsection{Staircase picture}
\label{ssec:example-staircase}

In \S~\ref{ssec:bimod-staircase}, we defined subgroups $F'_H$ and $F''_H$ of 
$\bigprofintensor{X_S} \vsV = \Fun (\graphB^{X_S}, \F_2)^{\cc}$, for each $H 
\supset \germ$. In this section, we describe the subgroups $F'_H$ and $F''_H$ 
in the notation of \S~\ref{ssec:example-notation}.

The groups $H \supset \germ$ all have the form $H = 2^m \Z_2$. For such an $H$, 
\begin{enumerate}
  \item $F'_H$ is the set of $H$-invariant magnetized tensors. For example,  
    \begin{equation}
      \label{eq:F-prime-elem}
      \begin{split}
        & \phantom{{}+} xy + (xxxy + xyxx) + (x^7 y + x^5yx^2 + x^3yx^4 + 
        xyx^6) \\
        & + \dotsb + \left(\sum_{\substack{i = 1\\ \text{$i$ odd}}}^{2^n} 
          x^{2^n - i} y x^{i-1}\right) + \dotsb
      \end{split}
    \end{equation}
    is an element of $F'_{2\Z_2}$. In fact, the sum of any (finite or infinite) 
    subset of these terms is also an element of $F'_{2 \Z_2}$.
  \item $F''_H$ is the kernel of the map $F'_H \to \Fun(\graphB^{H\backslash 
      {X_S}},\F_2)^H$, which is just restriction of functions. When $H = 2^m 
    \Z_2$, a series $\phi \in F'_H$ belongs to $F''_H$ if and only if every 
    pure tensor in its support is of length longer than $2^m$. For example, the 
    element in \eqref{eq:F-prime-elem} is \emph{not} in $F''_{2\Z_2}$. But if 
    we subtract (or equivalently, add) $xy$ to \eqref{eq:F-prime-elem}, the 
    result is in $F''_{2\Z_2}$.
\end{enumerate}

\subsection{The figure-eight knot}
\label{ssec:example-fig-eight}

For the rest of this section, we explain how $F'_H / F''_H$ in 
\S~\ref{ssec:example-staircase} is related to the sutured Heegaard Floer 
homology of a knot exterior and its covers. In this subsection we will describe 
these $3$-manifolds and their ``sutures,'' and in \S~\ref{ssec:example-sfh} we 
will recall some notation for sutured Heegaard Floer homology.

Let $\knotL = \figeight$ be the figure-eight knot.
Let $\mfd$ be the exterior, i.e.\ the complement of a tubular neighborhood $N$, 
of $\knotL$. It is a $3$-manifold whose boundary is a torus.
The boundary of a slice of $N$ is called a meridian. A section of $\partial N 
\to \knotL$ that is null-homologous in $\mfd$ is called a longitude.
As for any knot exterior, there is a unique homotopy class of parametrizations 
$\partial \mfd = \mer \times \longt$, with the property that every $\mer \times 
\set{q}$ is a meridian of $\knotL$, and every $\set{p} \times \longt$ is a 
longitude of $\knotL$. Fix such a parametrization and choose two distinct 
points $q_1, q_2 \in \longt$, and let $\Gamma \subset \partial \mfd$ be the 
union of $\mer \times \set{q_1}$ and $\mer \times \set{q_2}$. The pair $(\mfd, 
\Gamma)$ is an example of a \emph{sutured manifold}, and the two components of 
$\Gamma$ are called meridional sutures \cite[Example~2.4]{Juh06}.

For each integer $n$, there is a connected covering space of $\mfd$ whose deck 
group is cyclic of order $n$. It is unique up to isomorphism of covering 
spaces, and we call it the $n$th Alexander covering. Its boundary is 
parametrized by the product of the longitude and an $n$-fold covering circle of 
the $\mer$. When $n = 2^m$, we denote the corresponding Alexander covering by 
$\mfd_H$, where $H = 2^m \Z_2 \supset \germ$. The boundary of $\mfd_H$ is the 
preimage of the boundary of $\mfd$ under the covering map, and we denote the 
preimage of $\Gamma$ by $\Gamma_H$. Denote the two components of $\Gamma_H$ by 
$\mer_H \times \set{q_1}$ and $\mer_H \times \set{q_2}$. The pair $(\mfd_H, 
\Gamma_H)$ is also a sutured manifold. In fact $(\mfd,\Gamma)$ and the covers 
$(\mfd_H,\Gamma_H)$ are so-called ``balanced'' sutured manifolds 
\cite[Definition 2.2]{Juh06}.

\subsection{Sutured Heegaard Floer Homology}
\label{ssec:example-sfh}
Juhasz has defined \cite[Definition~7.6]{Juh06}, for a general balanced sutured 
$3$-manifold $(M,\Gamma \subset \partial M)$, an abelian group that we will 
denote by $\SFH(M,\Gamma;\Z)$ and call the \emph{sutured Heegaard Floer 
  homology} of $(M,\Gamma)$ over $\Z$.  It is defined as the homology of a 
chain complex $(\SFC(\Sigma,\alpha,\beta;\Z),\partial)$ of free abelian groups 
(denoted $\CF(\Sigma,\alpha,\beta)$ in \cite[Definition~7.1]{Juh06}) associated 
with a Heegaard decomposition of $N$ that respects the sutured manifold 
structure (\cite[Definitions~2.7 and 3.11]{Juh06}). The homology of the tensor 
product of $(\SFC(\Sigma,\alpha,\beta;\Z),\partial)$ and another ring $R$ is 
called the sutured Heegaard Floer homology over $R$ and denoted 
$\SFH(M,\Gamma;R)$ (cf.\ \cite[Remark~10.8]{Juh06}). When $R = \F_2$ we will 
simply write it as $\SFH(M,\Gamma)$ and call it the sutured Heegaard Floer 
homology of $(M,\Gamma)$. We will ignore the grading (\cite[\S~8]{Juh06}) on 
$\SFH$ and $\SFC$, and regard $\SFH(M,\Gamma)$ as an ungraded $\F_2$-vector 
space.

For a knot (or indeed, link) $L$ in a closed $3$-manifold $Y$, a standard 
observation (cf. for example \cite[Theorem 1.5]{Juh08}) is that any Heegaard 
decomposition of the pair $(Y,L)$ determines a Heegaard decomposition of the 
exterior $Y \setminus N(L)$ with its meridional sutures $\Gamma$, and an 
isomorphism
\begin{equation}
  \label{eq:sfh-hfk}
  \SFH(Y \setminus N(L),\Gamma) \cong \HFKh(Y,L).
\end{equation}
When $L$ is a nullhomologous knot, the right-hand side of \eqref{eq:sfh-hfk} is 
the ``knot Floer homology'' defined in \cite{OzsSza04:HFK,Ras03}. We will call 
it the knot Heegaard Floer homology of $L$. 

The following is a restatement of the theorem of \S~\ref{ssec:intro-ex}.

\begin{theorem*}
  Let $X$, $\germ$, $S$, $\idemring$, and $\vsV$ be as in 
  \S~\ref{ssec:example-ring-bimod} and \S~\ref{ssec:example-notation}.  For 
  each $H = 2^m \Z_2 \supset \germ$, let $F''_H \subset F'_H \subset 
  \bigprofintensor{X_S} \vsV$ be as in \S~\ref{ssec:example-staircase}. Let 
  $(M_H,\Gamma_H)$ be as in \S~\ref{ssec:example-fig-eight}. There is an 
  isomorphism of $\F_2$--vector spaces
  \begin{equation}
    \label{eq:first-part-of-theorem}
    F'_H/F''_H \cong \SFH(M_H,\Gamma_H).
  \end{equation}
\end{theorem*}

When $K = 2^{m+1} \Z_2$, the inclusions
\[
  F''_K \cap F'_H \subset F''_H \subset F'_H \subset F'_K
\] together with the isomorphisms \eqref{eq:first-part-of-theorem} for $H$ and 
$K$ imply that $\SFH(M_H,\Gamma_H)$ is a subquotient of $\SFH(M_K,\Gamma_K)$:
\[
  \begin{array}{c}
    \xymatrix{
      {F'_H}/({F''_K \cap F'_H}) \ar@{^(->}[r] \ar@{->>}[d] & F'_K/F''_K \\
      F'_H/F''_H
    }
  \end{array}
  \cong \begin{array}{c}
    \xymatrix{
      F'_H / (F''_K \cap F'_H) \ar@{^(->}[r] \ar@{->>}[d] & \SFH(M_K,\Gamma_K) 
      \\
      \SFH(M_H,\Gamma_H)
    }
  \end{array}
\]
The results of \cite{LipTre16} give a fairly natural subquotient relationship 
between $\SFH(M_K,\Gamma_K)$ and $\SFH(M_H,\Gamma_H)$. There are cyclic covers 
\begin{equation}
  \label{eq:YsubH}
  Y_K \xrightarrow{\pi} Y_H \to S^3
\end{equation}
branched over $L \subset S^3$, and, according to \eqref{eq:sfh-hfk}, there are 
isomorphisms
\[
  \HFKh(Y_H,L_H) \cong \SFH(M_H,\Gamma_H)\text{ and }\HFKh(Y_K,L_K) \cong 
  \SFH(M_K,\Gamma_K).
\]
\cite[Theorem 1]{LipTre16} says that there is a spectral sequence from the 
$\pi^*(\Spinct)$-graded piece of $\HFKh(Y_K,L_K)$ converging to the 
$\Spinct$-graded piece of $\HFKh(Y_H,L_H)$, at least when $\Spinct$ is a 
torsion $\Spinc$-structure and all knots involved have a Seifert surface of 
genus $\leq 2$. The figure-eight knot and its lifts in the branched covers have 
Seifert surfaces of genus $1$. Since the Alexander polynomial of $L:=4_1$ 
(which is $t^2 - 3t+1$) has no roots that are roots of unity, $H_1(Y_H;\Z)$ and 
$H_1(Y_K;\Z)$ are finite, and thus all $\Spinc$-structures on these manifolds 
are torsion.

We wonder whether there is a sense in which $\bigprofintensor{X_S} \vsV$ is the 
sutured Heegaard Floer homology of the ``sutured pro-$3$-manifold'' 
$(\varprojlim_H M_H,\varprojlim_H \Gamma_H)$. The $\Z_2$-bundle $\varprojlim_H 
M_H$ over $M$ is the $2$-adic completion of the Alexander cover of $M$. The 
boundary of the Alexander cover is a cylinder that unwraps the meridian of $M$; 
it is $\R \times \lambda$ in the notation of \S~\ref{ssec:example-fig-eight}.  
The boundary of $\varprojlim_H M_H$ is a solenoid times $\lambda$, and 
$\varprojlim_H \Gamma_H$ consists of two solenoids which cover the two 
meridional sutures of $\Gamma \subset \partial M$.

\subsection{\texorpdfstring{$M_H$}{The cyclic cover of the exterior} as a 
  mapping torus}
\label{ssec:example-mapping-torus}

The exterior $M$ and its $n$-fold cyclic covers are diffeomorphic to surface 
bundles over $S^1$. In other words, these manifolds are diffeomorphic to
\begin{equation}
  \label{eq:mapping-torus}
  ([0, 1] \times F) / \sim, \qquad (1, x) \sim (0, \psi (x)),
\end{equation}
where $F$ is a genus-$1$ surface with $1$ boundary component, and the monodromy 
$\psi \colon F \to F$ restricts to the identity on the boundary of $F$. For 
$M$, the fibers of the bundle are Seifert surfaces for $\knotL = \figeight$.

Like any genus-$1$ surface with $1$ boundary component, the oriented surface 
$F$ can be obtained from a disk by attaching two $1$-handles to its boundary.  
Fix such a presentation of $F$, and let $a$ and $b$ be oriented circles 
determined by the two $1$-handles respectively. The circles $a$ and $b$ 
intersect transversely at a single point. We denote by $\tau_a$ the 
right-handed Dehn twist along $a$, and by $\tau_b$ the right-handed Dehn twist 
along $b$. We also denote by $\tau_a^{-1}$ and $\tau_b^{-1}$ the corresponding 
left-handed Dehn twists. If the initial disk is equipped with a marked point 
$z$ in its boundary, then $a$ and $b$ may be distinguished. We do so by 
identifying $\tau_a$ with $\tau_\ell$, and $\tau_b$ with $\tau_m$, in 
\cite[Section~10.2]{LipOzsThu15}.

\begin{figure}[!htbp]
  \labellist
  \pinlabel $z$ at 50 0
  \pinlabel \textcolor{red}{$b$} at 50 50
  \pinlabel \textcolor{red}{$a$} at 230 50

  \endlabellist
  \includegraphics[scale=0.4]{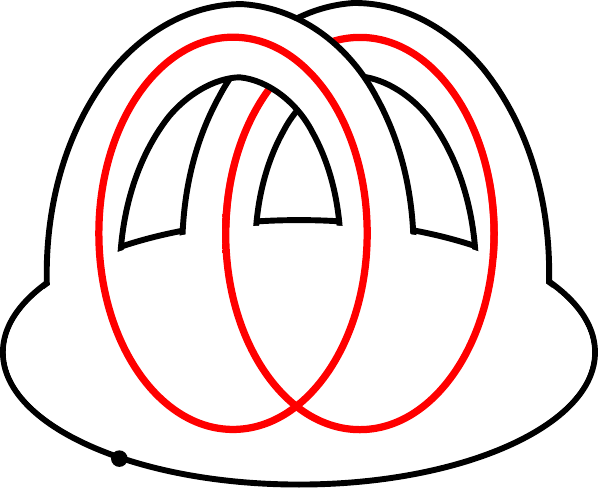}
  \caption{Disk with two handles.}
  \label{fig:disk-with-handles-labeled}
\end{figure}

All the same observations above apply also to the exterior of the trefoil and 
its finite cyclic covers. The exterior of the trefoil is diffeomorphic to 
\eqref{eq:mapping-torus} when $\psi = \tau_a \circ \tau_b$; the exterior of its 
mirror is diffeomorphic to \eqref{eq:mapping-torus} when $\psi = \tau_b^{-1} 
\circ \tau_a^{-1}$. The exterior of the figure-eight knot, what we have been 
denoting by $M$, is diffeomorphic to \eqref{eq:mapping-torus} when $\psi = 
\tau_a \circ \tau_b^{-1}$.  Note that the mapping torus associated to the 
cyclic permutation $\tau_b^{-1} \circ \tau_a$ is a diffeomorphic manifold.  
Since $L = 4_1$ is isotopic to its own mirror, the mapping tori associated to 
$\tau_b \circ \tau_a^{-1}$ and $\tau_a^{-1} \circ \tau_b$ are also 
diffeomorphic to $M$.

Now with this setup, we may also describe the cyclic $n$-fold covers of $M$ as 
mapping tori.
Fix $\phi = \tau_a \circ \tau_b^{-1}$. For $H = 2^m \Z_2 \supset \germ$, the 
manifold $M_H$ is diffeomorphic to \eqref{eq:mapping-torus} when $\psi = 
\phi^{2^m}$.

\subsection{\texorpdfstring{$Y_H$}{The branched cyclic cover} as an open book}
\label{ssec:example-open-book}

Since $\phi^{2^m}$ fixes pointwise the boundary $\partial F$ of $F$, the 
boundary of the mapping torus $([0, 1] \times F) / \sim$ is a parametrized 
torus. The diffeomorphism to $M_H$ parametrizes $\partial M_H$:
\[
  ([0, 1] / \sim) \times \partial F \stackrel{\sim}{\to} \partial M_H.
\]
Let $I_z$ be a small closed interval in $\partial F$ containing $z$. The 
diffeomorphism $([0, 1] \times F) / \sim \to M_H$ can be chosen to carry $([0, 
1] \times \partial I_z) / \sim$ to $\Gamma_H$. In other words, the two 
components of \[
  ([0, 1] / \sim) \times \partial I_z = ([0, 1] \times \partial I_z) / \sim
\]
are two meridians of our knot $L$.

If we attach a solid torus $D^2 \times S^1$ to $([0, 1] \times F) / \sim$ along
\[
  ([0, 1] / \sim) \times \partial F \cong \partial D^2 \times S^1,
\]
we obtain a manifold diffeomorphic to the cyclic cover $Y_H$ of $S^3$ branched 
along $\knotL = \figeight$ as in \eqref{eq:YsubH}. This diffeomorphism to $Y_H$ 
is an open-book decomposition of $Y_H$. The image of each $\set{p} \times F$ in 
$Y_H$ is a page of the open book, and the image of the core $\set{0} \times 
S^1$ of the solid torus $D^2 \times S^1$ is the binding $L_H$.

Since $Y_{\Z_2} \cong S^3$ is a $\Z$-homology sphere, $Y_H$ is the only 
$2^m$-fold cyclic cover of $S^3$ branched along $L$ for each $H = 2^m \Z_2 
\supset \germ$.

\subsection{Bordered Heegaard Floer theory}
\label{ssec:example-bhf}

The handle decomposition of $F$ in Figure~\ref{fig:disk-with-handles-labeled} 
is encoded by a pointed matched circle $\pmc$ in the sense of 
\cite[Definition~3.16]{LipOzsThu18}, and determines a differential graded (dg) 
algebra $\dga = \dga (\pmc)$ over an idempotent ring $\idemring (\pmc)$ 
(\cite[Definition~3.23]{LipOzsThu18}). For the handle decomposition of $F$ in 
Figure~\ref{fig:disk-with-handles-labeled}, $\idemring (\pmc)$ is exactly the 
$\idemring$ in \S~\ref{ssec:example-ring-bimod}.

In bordered Heegaard Floer theory \cite{LipOzsThu18, LipOzsThu15}, there are 
four species of $(\dga, \dga)$-bimodules; they are called ``type-\AA, \DD, \DA, 
and \AD'' bimodules. (These remarks apply to our algebra $\dga (\pmc)$ and to a 
more general class of $\Ainf$-algebras.) None of these four are traditional dg 
bimodules, but their associated derived categories are equivalent 
\cite[Proposition~2.3.18]{LipOzsThu15} and also equivalent to that of dg 
bimodules \cite[Proposition~2.4.1]{LipOzsThu15}.  There is a Hochschild 
homology for each of these species of bimodules, which is a functor valued in 
vector spaces.  The explicit model for the Hochschild homology of type-\DA{} 
bimodules $\HHbox$ \cite[\S~2.3]{LipOzsThu15} will be useful to us.  Given a 
diffeomorphism $\psi \colon (F, z) \to (F, z)$ that restricts to the identity 
on the boundary $\partial F$, bordered Heegaard Floer theory determines a 
quasi-isomorphism class of each type of these bimodules. We denote by
\[
  \CFDAh (\psi) =
  \prescript{\dga}{}{\CFDAh (\psi)}_{\dga}
\]
a representative of the quasi-isomorphism type of type-\DA{} bimodules 
associated to $\psi$. (The left superscript indicates that it is of type-D on 
the left, and the right subscript indicates that it is of type-A on the right.)

\begin{theorem*}[{\cite[Theorem~7]{LipOzsThu15}}]
  \label{thm:bfh-hh}
  Let $\psi \colon (\surf, \basept) \to (\surf, \basept)$ be a diffeomorphism 
  that restricts to the identity on the boundary $\bdy \surf$, giving an 
  open-book decomposition with binding $L \subset Y$. The Hochschild homology 
  of the bimodule $\CFDAh (\psi)$ computes the knot Floer homology $\HFKh (Y, 
  L)$.
\end{theorem*}

\begin{remark*}
  The theory of bordered--sutured Floer homology by Zarev \cite{Zar11} provides 
  an alternative perspective on Theorem~\ref{thm:bfh-hh}.
  Consider the $3$-manifold $[0, 1] \times \partial F$, without identifying the 
  $(1, x)$ and $(0, \psi (x))$ via the monodromy.
  In this framework, we partition the boundary of this manifold into three 
  parts: the ``top'' $\set{1} \times \partial F$, the ``bottom'' $\set{0} 
  \times \partial F$, and the ``side'' $[0, 1] \times \bdy F$. There are 
  ``sutures'' $[0, 1] \times \bdy I_z$ on the ``side'' boundary, consisting of 
  two disjoint intervals.
  Then $F$ parametrizes both the top $\set{1} \times \partial F$ and the bottom 
  $\set{0} \times \partial F$, but in different ways, according to the 
  monodromy $\psi$. These parametrizations allow the top and the bottom to be 
  ``glued'' to other $3$-manifolds. If we glue the top to the bottom, then we 
  obtain the manifold $([0, 1] \times F) / \sim$ that we have seen before, 
  which is diffeomorphic to $Y \setminus N (L)$.  The sutures glue up to $([0, 
  1] / \sim) \times \bdy I_z$, which may be identified with the meridional 
  sutures $\Gamma$ as before.

  $\CFDAh (\psi)$ may then be understood as the bordered--sutured Floer 
  bimodule associated to the manifold $[0, 1] \times \partial F$. Then the 
  statement of Theorem~\ref{thm:bfh-hh} becomes that the Hochschild homology of 
  $\CFDAh (\psi)$ computes $\SFH (Y \setminus N (L), \Gamma)$, which is 
  isomorphic to $\HFKh (Y, L)$ by \eqref{eq:sfh-hfk}.
\end{remark*}

A consequence of Theorem~\ref{thm:bfh-hh} is that we can compute $\HFKh (Y_H, 
L_H)$ by computing the Hochschild homology of $\CFDAh (\psi)$ with $\psi = 
\phi^{2^m}$.  We now further break down this computation:

\begin{theorem*}[{\cite[Theorem~5]{LipOzsThu15}}]
  \label{thm:bfh-comp}
  Let $\sbdiff_1, \sbdiff_2 \colon (\surf, \basept) \to (\surf, \basept)$ be 
  diffeomorphisms that restrict to the identity on the boundary $\bdy \surf$.  
  We have a quasi-isomorphism
  \[
    \CFDAh (\sbdiff_2 \circ \sbdiff_1) \qisom
    \CFDAh (\sbdiff_1) \boxtensor \CFDAh (\sbdiff_2),
  \]
  where $\boxtensor$ denotes the box tensor 
  \cite[Definition~2.3.9]{LipOzsThu15}.
\end{theorem*}

The box tensor $\boxtensor$ \cite[Definition~2.3.9]{LipOzsThu15} is a specific 
model for the derived tensor product for type-\DA{} bimodules.
The standard equivalence between the derived category of type-\DA{} bimodules 
and the derived category of dg bimodules intertwines the box tensor product 
$\boxtensor$ and the derived tensor product 
\cite[Corollary~2.3.23]{LipOzsThu15}.

In Zarev's framework \cite{Zar11}, Theorem~\ref{thm:bfh-comp} is simply the 
gluing theorem, where we glue together the manifolds $[0, 1] \times \bdy \surf$ 
associated to $\phi_1$ and $\phi_2$.

Putting everything together, for $H = 2^m \Z_2 \supset \germ$, we have
\begin{equation}
  \label{eq:hfk-hh-2m}
  \SFH (M_H, \Gamma_H) \cong
  \HFKh (Y_H, L_H) \cong
  \HHbox (\underbrace{\CFDAh (\sbdiff) \boxtensor \dotsb \boxtensor \CFDAh 
    (\sbdiff)}_{2^m \text{ times}}).
\end{equation}

Thus, in the case of the figure-eight knot, we have reduced the problem to 
computing $\CFDAhsbdiff$, where $\sbdiff = \twistl \comp \twistm^{-1}$ is a 
diffeomorphism of the base-pointed genus-$1$ surface $(\surf, \basept)$ that 
restricts to the identity on the boundary $\bdy \surf$.  Again using 
Theorem~\ref{thm:bfh-comp}, we have
\[
  \CFDAhsbdiff = \CFDAh (\twistl \comp \twistm^{-1}) \qisom
  \CFDAh (\twistm^{-1}) \boxtensor \CFDAh (\twistl).
\]
The dg algebra $\dga (\pmc)$ and the bimodules $\CFDAh (\twistm^{-1})$ and 
$\CFDAh (\twistl)$ are explicitly computed in \cite[Sections~3.3 and 
10.2]{LipOzsThu15}, which we recall below.

\subsection{The torus algebra}
\label{ssec:example-dga}

Let $\pmc$ be the pointed matched circle associated with the handle 
decomposition of $\surf$ in Figure~\ref{fig:disk-with-handles-labeled}.
The algebra $\dga (\pmc)$ and its idempotent subring $\idemring (\pmc)$ 
decompose as direct products by the ``strands grading'',
\begin{equation}
  \label{eq:dga-strands-gr}
  \begin{aligned}
    \dga (\pmc) &= \dga (\pmc, -1) \times \dga (\pmc, 0) \times \dga (\pmc, 
    1),\\
    \idemring (\pmc) &= \idemring (\pmc, -1) \times \idemring (\pmc, 0) \times 
    \idemring (\pmc, 1),
  \end{aligned}
\end{equation}
where $\dga (\pmc, -1) = \idemring (\pmc, -1) = \Ftwo \gen{\idem_{\emptyset}}$ 
and the inclusion $\idemring (\pmc, 1) = \Ftwo \gen{\idem_{\set{0,1}}} \to \dga 
(\pmc, 1)$ is a quasi-isomorphism.  A consequence is that all quasi-invertible 
bimodules over $\dga (\pmc, -1)$ and $\dga (\pmc, 1)$ are quasi-isomorphic to 
the identity bimodule $\Ftwo$.

The underlying vector space of $\dga (\pmc, 0)$ is $8$-dimensional over 
$\Ftwo$. A basis for $\dga (\pmc, 0)$ is
\begin{equation}
  \label{eq:dga-basis}
  \idem_{\set{0}}, \idem_{\set{1}},
  \strand_1, \strand_2, \strand_3, \strand_{12}, \strand_{23}, \strand_{123},
\end{equation}
where $\idem_{\set{0}}$ and $\idem_{\set{1}}$ are idempotents with 
$\idem_{\set{0}} \idem_{\set{1}} = \idem_{\set{1}} \idem_{\set{0}} = 0$. We 
denote the ring $\Ftwo \gen{\idem_{\set{0}}, \idem_{\set{1}}}$ by $\idemring 
(\pmc, 0)$. The other six elements
have idempotent compatibility conditions
\begin{align*}
  \strand_1 &= \idem_{\set{0}} \strand_1 \idem_{\set{1}}, &
  \strand_2 &= \idem_{\set{1}} \strand_2 \idem_{\set{0}}, &
  \strand_3 &= \idem_{\set{0}} \strand_3 \idem_{\set{1}}, \\
  \strand_{12} &= \idem_{\set{0}} \strand_{12} \idem_{\set{0}}, &
  \strand_{23} &= \idem_{\set{1}} \strand_{23} \idem_{\set{1}}, &
  \strand_{123} &= \idem_{\set{0}} \strand_{123} \idem_{\set{1}}.
\end{align*}

The differential on $\dga (\pmc, 0)$ is identically zero, while the 
multiplication between two of the $\strand$'s is also zero except for
\[
  \strand_1 \strand_2 = \strand_{12}, \qquad
  \strand_2 \strand_3 = \strand_{23}, \qquad
  \strand_1 \strand_{23} = \strand_{123}, \qquad
  \strand_{12} \strand_3 = \strand_{123}.
\]

Besides the strands grading \eqref{eq:dga-strands-gr}, we will not make use of 
the other gradings on $\idemring (\pmc)$, $\dga (\pmc)$, and $\CFDAhsbdiff$.

\subsection{The bimodules \texorpdfstring{$\CFDAh (\twistm^{-1}, 0)$ and 
    $\CFDAh (\twistl, 0)$}{associated to Dehn twists in degree 0}}
\label{ssec:example-bimods}

As both the bimodules
\[
  \prescript{\dga (\pmc)}{}{\CFDAh (\twistm^{-1})}_{\dga (\pmc)}
  \qquad \text{and} \qquad
  \prescript{\dga (\pmc)}{}{\CFDAh (\twistl)}_{\dga (\pmc)}
\]
are bimodules over $\dga (\pmc)$, which decomposes by the strands grading, we 
have a corresponding direct sum decomposition
\begin{gather*}
  \CFDAh (\twistm^{-1}) =
  \bigdirsum_{i \in \set{-1, 0, 1}}
  \prescript{\dga (\pmc, i)}{}{\CFDAh (\twistm^{-1}, i)}_{\dga (\pmc, i)},\\
  \CFDAh (\twistl) =
  \bigdirsum_{i \in \set{-1, 0, 1}}
  \prescript{\dga (\pmc, i)}{}{\CFDAh (\twistl, i)}_{\dga (\pmc, i)}.
\end{gather*}
Since $\CFDAh (\psi, i)$ is quasi-isomorphic to the identity bimodule $\Ftwo$ 
for $i \in \set{-1, 1}$ and any $\psi$, we focus only on $\CFDAh (\twistm^{-1}, 
0)$ and $\CFDAh (\twistl, 0)$.

The vector space underlying the bimodule $\CFDAh (\twistm^{-1}, 0)$ is 
generated over $\Ftwo$ by three generators $\genpmi$, $\genqmi$, and $\genrmi$,
with compatibility conditions
\[
  \genpmi = \idem_0 \cdot \genpmi \cdot \idem_0, \qquad
  \genqmi = \idem_1 \cdot \genqmi \cdot \idem_1, \qquad
  \genrmi = \idem_1 \cdot \genrmi \cdot \idem_0.
\]
The vector space underlying the bimodule $\CFDAh (\twistl, 0)$ is generated 
over $\Ftwo$ by three generators $\genpl$, $\genql$, and $\gensl$,
\[
  \genpl = \idem_0 \cdot \genpl \cdot \idem_0, \qquad
  \genql = \idem_1 \cdot \genql \cdot \idem_1, \qquad
  \gensl = \idem_0 \cdot \gensl \cdot \idem_1.
\]

To endow these $(\idemring (\pmc), \idemring (\pmc))$-bimodules with a 
type-\DA{} $(\dga (\pmc), \dga (\pmc))$-bimodule structure, it remains to 
specify the \emph{type-\DA{} structure maps}. In general, given a dg algebra 
$\dga$ over a ground ring $\idemring$, a type-\DA{} bimodule 
$\prescript{\dga}{}{M}_{\dga}$ is specified by $(\idemring, \idemring)$-linear 
maps
\[
  \DAmap_{1+j}^1 \colon M \tensor \dga [1]^{\tensor j} \to \dga [1] \tensor M,
\]
for $j \geq 0$, satisfying the compatibility conditions in 
\cite[Definition~2.2.43]{LipOzsThu15}.
The structure maps for $\CFDAh (\twistm^{-1}, 0)$ and $\CFDAh (\twistl, 0)$ are 
given by Figure~\ref{fig:DAbimodules}.

\begin{figure}[!htbp]
  \begin{tikzpicture}[y=54pt,x=1in,scale=0.9]
    \node at (-0.5,2) (pmi) {$\genpmi$};
    \node at (1.5,2) (qmi) {$\genqmi$};
    \node at (0.5,0) (rmi) {$\genrmi$};
    \node at (2.5,2) (ql) {${\genql}$};
    \node at (4.5,2) (pl) {${\genpl}$};
    \node at (3.5,0) (sl) {${\gensl}$};

    \draw[->] (rmi) [bend left=15] to node[above,sloped] 
    {$\scriptstyle{\rho_{23}\tensor (\rho_3)}$} (qmi);

    \draw[->] (pmi) [bend left=15] to node[above,sloped] 
    {$\scriptstyle{\rho_1\tensor (\rho_1)+\rho_{123}\tensor(\rho_{123})}$}  
    (qmi);
    \draw[->] (pmi) [bend right=15] to node[below,sloped]  
    {$\scriptstyle{\rho_3\tensor () + \rho_1\tensor (\rho_{12})}$} (rmi);
    \draw[->] (pmi) [loop] to node[above] 
    {$\scriptstyle{\rho_{12}\tensor(\rho_{123},\rho_2)}$} (pmi);
    \draw[->] (qmi) [bend left=15] to node[below,sloped] 
    {$\scriptstyle{\rho_2\tensor(\rho_{23},\rho_2)}$} (pmi);
    \draw[->] (qmi) [loop] to node[above] 
    {$\scriptstyle{\rho_{23}\tensor(\rho_{23})}$} (qmi);
    \draw[->] (rmi) [bend right=15] to node[above,sloped] 
    {$\scriptstyle{\rho_2\tensor (\rho_3,\rho_2)}$} (pmi);
    \draw[->] (qmi) [bend left=15] to node[below,sloped] 
    {$\scriptstyle{1\tensor (\rho_2)}$} (rmi);

    \draw[->] (sl) [bend left=15] to node[above,sloped] {$\scriptstyle{1\tensor 
        (\rho_2)}$} (pl);
    \draw[->] (ql) [bend left=15] to node[above,sloped] 
    {$\scriptstyle{\rho_2\tensor (\rho_2,\rho_1)}$} (sl);
    \draw[->] (ql) [loop] to node[above]  {$\scriptstyle{\rho_{23}\tensor 
        (\rho_2,\rho_{123})}$}  (ql);
    \draw[->] (ql) [bend right=15] to node[below,sloped] 
    {$\scriptstyle{\rho_2\tensor(\rho_2,\rho_{12})}$}  (pl);
    \draw[->] (pl) [bend left=15] to node[below,sloped] 
    {$\scriptstyle{\rho_{12}\tensor(\rho_1)}$} (sl);
    \draw[->] (pl) [bend right=15] to node[above,sloped] 
    {$\scriptstyle{\rho_3\tensor(\rho_3)+\rho_{123}\tensor(\rho_{123})}$} (ql);
    \draw[->] (pl) [loop] to node[above] 
    {$\scriptstyle{\rho_{12}\tensor(\rho_{12})}$} (pl);
    \draw[->] (sl) [bend left=15] to node[below,sloped] 
    {$\scriptstyle{\rho_1\tensor ()+\rho_3\tensor(\rho_{23})}$} (ql);
  \end{tikzpicture}
  \caption{Left: The type-\DA{} bimodule $\CFDAh (\twistm^{-1}, 0)$.  Right: 
    The type-\DA{} bimodule $\CFDAh (\twistl, 0)$.}
  \label{fig:DAbimodules}
\end{figure}

Let us explain how to read the structure maps from the diagrams of 
Figure~\ref{fig:DAbimodules}, starting with the unique arrow from $\genrmi$ to 
$\genqmi$ and the unique arrow from $\genqmi$ to $\genpmi$ in the left diagram.  
The label $\strand_{23} \tensor \strand_3$ on $\genrmi \to \genqmi$ indicates 
that $\strand_{23} \tensor \genqmi$ appears as a term in $\DAmap_2^1 (\genrmi 
\tensor \strand_3)$. Since $\blank \tensor \strand_3$ does not appear as a 
label on any other arrow out of $\genrmi$, there are no other contributions, 
and
\[
  \DAmap_2^1 (\genrmi \tensor \strand_3) = \strand_{23} \tensor \genqmi.
\]
On $\genqmi \to \genpmi$, the label $\strand_2 \tensor (\strand_{23}, 
\strand_2)$ indicates that $\strand_2 \tensor \genpmi$ appears as a term in 
$\DAmap_3^1 (\genqmi \tensor \strand_{23} \tensor \strand_2)$. Again, there are 
no other contributions, and so
\[
  \DAmap_3^1 (\genqmi \tensor \strand_{23} \tensor \strand_2) = \strand_2 
  \tensor \genpmi.
\]
In general, the $j$-tuple of $\strand$'s to the right of the $\tensor$ symbol 
of a summand of a label denotes the $j$ input algebra elements of 
$\DAmap_{1+j}^1$.  The $\strand$ or $1$ to the left of the $\tensor$ symbol 
denotes the output algebra element of a contribution to $\DAmap_{1+j}^1$.
For example, on $\genpmi \to \genrmi$, the label $\strand_3 \tensor () + 
\strand_1 \tensor (\strand_{12})$ indicates contributions to two different 
structure maps $\DAmap_1^1$ and $\DAmap_2^1$. There are no other contributions 
to $\DAmap_1^1 (\genp)$ and $\DAmap_2^1 (\genp \tensor \strand_{12})$, and so
\[
  \DAmap_1^1 (\genp) = \strand_3 \tensor \genr, \qquad \DAmap_2^1 (\genp 
  \tensor \strand_{12}) = \strand_1 \tensor \genr.
\]
In addition to the information encoded by the diagrams, the structure maps 
always obey the \emph{strictly unital} condition: For all $x$ in the bimodule,
\begin{equation}
  \label{eq:strictly-unital}
  \begin{aligned}
    \DAmap_2^1 (x \tensor 1) & = 1 \tensor x,\\
    \DAmap_{1+j}^1 (x \tensor a_1 \tensor \dotsb \tensor a_j) & = 0 \text{ 
      whenever } j > 1 \text{ and some } a_k \in \idemring (\pmc, 0).
  \end{aligned}
\end{equation}
Since, for $j \geq 3$, there are no summands in any label with a $j$-tuple of 
$\rho$'s to the right of the $\tensor$ symbol, only $\DAmap_1^1$, $\DAmap_2^1$, 
and $\DAmap_3^1$ are not identically zero:
\[
  \DAmap_{1+j}^1 \equiv 0 \quad \text{for} \quad j \geq 3.
\]

\subsection{The bimodule \texorpdfstring{$\CFDAh (\twistm^{-1}, 0) \boxtensor 
    \CFDAh (\twistl, 0)$}{associated to the exterior of the figure-eight knot 
    and its Seifert surface}}
\label{ssec:example-tensor}

We are now ready to compute $\CFDAh (\twistm^{-1}, 0) \boxtensor \CFDAh 
(\twistl, 0)$.

For short, we will write
\[
  \bimodmerinv := \CFDAh (\twistm^{-1}, 0), \qquad
  \bimodlong := \CFDAh (\twistl, 0)
\]
respectively.  The underlying $(\idemring (\pmc, 0), \idemring (\pmc, 
0))$-bimodule of $\bimodmerinv \boxtensor \bimodlong$ is defined 
\cite[Definitions~2.3.2, 2.3.9]{LipOzsThu15} to be
\(
  \bimodmerinv \tensor_{\idemring (\pmc, 0)} \bimodlong.
\)
The only pure tensors that do not vanish are
$\genpmi \tensor_{\idemring (\pmc, 0)} \genpl$,
$\genpmi \tensor_{\idemring (\pmc, 0)} \gensl$,
$\genqmi \tensor_{\idemring (\pmc, 0)} \genql$,
$\genrmi \tensor_{\idemring (\pmc, 0)} \genpl$, and
$\genrmi \tensor_{\idemring (\pmc, 0)} \gensl$.
We denote them by $\genpp$, $\genps$, $\genqq$, $\genrp$, and $\genrs$ for 
short. They form an $\Ftwo$-basis for $\bimodmerinv \boxtensor \bimodlong$. The 
$(\idemring (\pmc, 0), \idemring (\pmc, 0))$-bimodule structure is given by
\begin{gather*}
  \genpp = \idem_0 \cdot \genpp \cdot \idem_0, \qquad
  \genps = \idem_0 \cdot \genps \cdot \idem_1, \qquad
  \genqq = \idem_1 \cdot \genqq \cdot \idem_1, \\
  \genrp = \idem_1 \cdot \genrp \cdot \idem_0, \qquad
  \genrs = \idem_1 \cdot \genrs \cdot \idem_1.
\end{gather*}

Let us recall some information about the type-\DA{} structure maps of a box 
tensor of two type-\DA{} structures from \cite[Section~2.3]{LipOzsThu15}.
Let $M$ and $N$ be two type-\DA{} bimodules that can be tensored together, with 
structure maps
$\DAmap_{1+j}^{1, M}$ and $\DAmap_{1+j}^{1, N}$ respectively.
The structure maps
$\DAmap_{1+j}^{1, \boxtensor}$ of $M \boxtensor N$ are
given by an expression involving the tensor algebra $T^* (\dga [1])$ and maps 
\[
  \DAmap^{k, N} \colon N \tensor T^* (\dga [1]) \to \dga [1]^{\tensor k} 
  \tensor N,
\]
defined by an inductive process from the structure maps $\DAmap_{1+j}^{1, N}$.
(Some details about this inductive process are explained around 
\eqref{eq:DAmap-genpl} below.)
The maps $\DAmap_{1+j}^{1, \boxtensor}$ are determined by the rule:
\begin{equation}
  \label{eq:DAmap-box}
  \sum_{j=0}^\infty \DAmap_{1+j}^{1, \boxtensor} := \sum_{k=0}^\infty 
  (\DAmap_{1+k}^{1, M} \tensor \id_N) \comp
  (\id_M \tensor \DAmap^{k,N}).
\end{equation}

Informally, the value of \eqref{eq:DAmap-box} on $x \tensor y \tensor a_1 
\tensor \dotsb \tensor a_j$
is the sum of all possible results of applying, for some $k \geq 0$, the 
composition of $k$ iterations of $\DAmap_{1+j_i}^{1, N}$ to $y$, each iteration 
consuming $j_i \geq 0$ of the input algebra elements $a_1, \dotsc, a_j$, such 
that all $j$ input algebra elements are consumed after the $k$ iterations, and 
then feeding $x$ and the $k$ output algebra elements of $\DAmap_{1+j_i}^{1, N}$ 
to exactly one iteration of $\DAmap_{1+k}^{1, M}$.

All of the structure maps $\DAmap_{1+j}^{1, \boxtensor}$ for $\bimodmerinv 
\boxtensor \bimodlong$ are then determined by a standard computation (which we 
expand on below), reported here in a format similar to that of 
Figure~\ref{fig:DAbimodules}, but with the targets of the arrows inserted as a 
middle tensor factor between the input and output $\strand$'s:
\begin{equation}
  \label{eq:DAbimod-box-final}
  \begin{aligned}
    \genpp & \mapsto
    \strand_{12} \tensor \genpp \tensor (\strand_{123}, \strand_{2}, 
    \strand_{12})
    + \strand_{12} \tensor \genps \tensor (\strand_{123}, \strand_2, 
    \strand_1)\\
    & \phantom{{} \mapsto {}}
    + \strand_{123} \tensor \genqq \tensor (\strand_{123})
    + \strand_{1} \tensor \genrs \tensor (\strand_{1})\\
    & \phantom{{} \mapsto {}}
    + \strand_{1} \tensor \genrp \tensor (\strand_{12})
    + \strand_{3} \tensor \genrp \tensor (), \\
    \genps & \mapsto
    1 \tensor \genpp \tensor (\strand_2)
    + \strand_1 \tensor \genqq \tensor ()
    + \strand_3 \tensor \genrs \tensor (), \\
    \genqq & \mapsto
    \strand_{2} \tensor \genpp \tensor (\strand_2, \strand_{123}, \strand_2, 
    \strand_{12})
    + \strand_2 \tensor \genps \tensor (\strand_2, \strand_{123}, \strand_2, 
    \strand_1)\\
    & \phantom{{} \mapsto {}}
    + \strand_{23} \tensor \genqq \tensor (\strand_2, \strand_{123})
    + 1 \tensor \genrp \tensor (\strand_2, \strand_{12})\\
    & \phantom{{} \mapsto {}}
    + 1 \tensor \genrs \tensor (\strand_2, \strand_1), \\
    \genrp & \mapsto
    \strand_{2} \tensor \genpp \tensor (\strand_3, \strand_2, \strand_{12})
    + \strand_2 \tensor \genps \tensor (\strand_3, \strand_2, \strand_1)
    + \strand_{23} \tensor \genqq \tensor (\strand_3), \\
    \genrs & \mapsto
    \strand_2 \tensor \genpp \tensor (\strand_{23}, \strand_2, \strand_{12})
    + \strand_2 \tensor \genps \tensor (\strand_{23}, \strand_2, \strand_1)\\
    & \phantom{{} \mapsto {}}
    + \strand_{23} \tensor \genqq \tensor (\strand_{23})
    + 1 \tensor \genrp \tensor (\strand_{2}).
  \end{aligned}
\end{equation}

Before we explain how to complete this computation systematically,
let us illustrate how the second summand in the row starting with $\genrp$ is 
obtained. The meaning of this summand is that $\strand_2 \tensor \genps$ 
appears as a term in $\DAmap_4^{1,\boxtensor} (\genrp \tensor \strand_3 \tensor 
\strand_2 \tensor \strand_1)$.

In the diagram for $\bimodlong = \CFDAh (\twistl, 0)$, there is a term 
$\strand_3 \tensor (\strand_3)$ in the label on $\genpl \to \genql$.  This term 
in the label means that
\(
  \strand_3 \tensor \genql
\)
appears as a term in
\(
  \DAmap_2^{1, \bimodlong} (\genpl \tensor \strand_3).
\)
Next, there is a term $\strand_2 \tensor (\strand_2, \strand_1)$ in the label 
on $\genql \to \gensl$, meaning that
\(
  \strand_2 \tensor \gensl
\)
appears as a term in
\(
  \DAmap_3^{1, \bimodlong} (\genql \tensor \strand_2 \tensor \strand_1).
\)
Combining these $k = 2$ iterations of $\DAmap_{1+j_i}^{1, \bimodlong}$ into one 
operation, we may write that
\(
  \strand_3 \tensor \strand_2 \tensor \gensl
\)
appears as a term in the value of
\(
  \DAmap^{k,\bimodlong} = \DAmap^{2,\bimodlong}
\)
on
\(
  \genpl \tensor \strand_3 \tensor \strand_2 \tensor \strand_1
\):
\begin{equation}
  \label{eq:DAmap-genpl}
  \DAmap^{2,\bimodlong} (\genpl \tensor \strand_3 \tensor \strand_2 \tensor 
  \strand_1) = \strand_3 \tensor \strand_2 \tensor \gensl + \dotsb.
\end{equation}
In fact, there are no other terms:
\(
  \DAmap^{2,\bimodlong} (\genpl \tensor \strand_3 \tensor \strand_2 \tensor 
  \strand_1) = \strand_3 \tensor \strand_2 \tensor \gensl
\).

Meanwhile, in the diagram for $\bimodmerinv = \CFDAh (\twistm^{-1}, 0)$, there 
is a term $\strand_2 \tensor (\strand_3, \strand_2)$ in the label on $\genrmi 
\to \genpmi$. This term in the label means that
\(
  \strand_2 \tensor \genpmi
\)
appears as a term in
\(
  \DAmap_3^{1,\bimodmerinv} (\genrmi \tensor \strand_3 \tensor \strand_2).
\)
Again, there are no other terms:
\(
  \DAmap_3^{1,\bimodmerinv} (\genrmi \tensor \strand_3 \tensor \strand_2)
  = \strand_2 \tensor \genpmi.
\)

The computations so far, i.e.\
\begin{equation*}
  \DAmap_3^{1, \bimodmerinv} (\genrmi \tensor \strand_3 \tensor \strand_2) = 
  \strand_2 \tensor \genpmi, \qquad
  \DAmap^{2, \bimodlong} (\genpl \tensor \strand_3 \tensor \strand_2 \tensor 
  \strand_1) = \strand_3 \tensor \strand_2 \tensor \gensl,
\end{equation*}
give
the value of the $k = 2$ summand of the right-hand side of \eqref{eq:DAmap-box} 
on $\genrp \tensor \strand_3 \tensor \strand_2 \tensor \strand_1$.
Specifically, the value is $\strand_2 \tensor \genps$.
Because there are $3$ $\strand$'s in $\genrp \tensor \strand_3 \tensor 
\strand_2 \tensor \strand_1$, it is acted upon by the $j = 3$ summand of the 
left-hand side of \eqref{eq:DAmap-box},
and we obtain a contribution to the value of $\DAmap_{1+j}^{1,\boxtensor} = 
\DAmap_4^{1,\boxtensor}$ on $\genrp \tensor \strand_3 \tensor \strand_2 \tensor 
\strand_1$:
\[
  \DAmap_4^{1, \boxtensor} (\genrp \tensor \strand_3 \tensor \strand_2 \tensor 
  \strand_1) = \strand_2 \tensor \genps + \dotsb.
\]
In fact, there are no other terms, but to verify this, one must complete the 
more systematic computation outlined below.

The algorithm to compute $\DAmap_{1+j}^{1, \boxtensor}$ is as follows. Fix one 
of the $5$ generators $x_1 y_1 = x_1 \tensor y_1 \in \set{\genpp, \genps, 
  \genqq, \genrp, \genrs}$ of $\bimodmerinv \boxtensor \bimodlong$.
Let $\outerloop$ be the set of ordered pairs (arrow, summand of label) out of 
$x_1$ in the diagram for $\bimodmerinv$. Let $\outerloopp$ be the union of 
$\outerloop$ with the singleton set $\set{(x_1 \to x_1, 1 \tensor 1)}$. This 
additional element is not displayed in the diagram for $\bimodmerinv$ but is 
required to account for the strictly unital condition 
\eqref{eq:strictly-unital}.
For each element of $\outerloop$,
write $(a_1, \dotsc, a_k)$ for the sequence of $\strand$'s that appear to the 
right of the $\tensor$ sign, write $b$ for the $\strand$ that appears to the 
left of the $\tensor$ sign, and write $x_2$ for the target of the arrow.
This means that
\begin{equation}
  \label{eq:DAmap-box-tensor-B}
  \DAmap_{1+k}^{1,\bimodmerinv} (x_1 \tensor a_1 \tensor \dotsb \tensor a_k) = 
  b \tensor x_2 + \dotsb.
\end{equation}
For the element of $\outerloopp \setminus \outerloop$, note that 
\eqref{eq:DAmap-box-tensor-B} also holds, with $(k, (a_1 \tensor \dotsb \tensor 
a_k), b) = (1, (1), 1)$, and $x_2 = x_1$.

If $k > 0$, find all length-$k$ sequences of pairs (arrow, summand of label) 
that have the following properties:
\begin{itemize}
  \item The sequence of arrows is a directed path starting at $y_1$;
  \item For $1 \leq i \leq k$, the summand of the label on the $i$-th arrow has 
    $a_i$ appearing to the left of the $\tensor$ sign.
\end{itemize}
For each such length-$k$ sequence of pairs (arrow, summand of label), let 
$y_{k+1}$ be the terminal vertex of the directed path. Concatenate the $k$ 
sequences of $\strand$'s to the right of the $\tensor$ signs associated to 
these pairs to obtain a single sequence $(c_1, \dotsc, c_j)$.
Record the contribution to the value of $\DAmap^{k,\bimodlong} (y_1 \tensor c_1 
\tensor \dotsb \tensor c_j)$ as in \eqref{eq:DAmap-genpl}:
\begin{equation}
  \label{eq:DAmap-box-tensor-A}
  \DAmap^{k,\bimodlong} (y_1 \tensor c_1 \tensor \dotsb \tensor c_j) = a_1 
  \tensor \dotsb \tensor a_k \tensor y_{k+1} + \dotsb.
\end{equation}
If $k = 0$, the only sequence of pairs is the empty sequence, and we record 
$\DAmap^{0,\bimodlong} (y_1) = y_1$.
Combine \eqref{eq:DAmap-box-tensor-B} and \eqref{eq:DAmap-box-tensor-A} to 
obtain:
\[
  \DAmap_{1+j}^{1,\boxtensor} (x_1 y_1 \tensor c_1 \tensor \dotsb \tensor c_j) 
  = b \tensor x_2 y_{k+1} + \dotsb.
\]

After completing this process
for every element of $\outerloopp$,
we will have computed all contributions to $\DAmap^{1,\boxtensor}$ that have 
$x_1 y_1$ appearing as a tensor factor in its input.
After completing this process for every $x_1 y_1$, we will have computed all 
contributions to $\DAmap^{1,\boxtensor}$.

\subsection{The differential of the Hochschild chain complex}
\label{ssec:example-ch}

We are almost ready to compute $\HFKh (Y_H, L_H)$ for $H = 2^m \Z_2 \supset 
\germ$.  It is the same, by \eqref{eq:hfk-hh-2m}, as computing
\[
  \HHbox (\underbrace{\CFDAh (\sbdiff) \boxtensor \dotsb \boxtensor \CFDAh 
    (\sbdiff)}_{2^m \text{ times}}),
\]
where $\HHbox$ is the Hochschild homology defined for type-\DA{} bimodules in 
\cite[\S~2.3.5]{LipOzsThu15}.
The Hochschild chain complex is denoted $\CHbox$.

\begin{proposition*}
  \label{prop:comp-diff-zero}
  For all $n$, the differential of the chain complex
  \[
    \CHbox (\underbrace{\CFDAh (\sbdiff, 0) \boxtensor \dotsb \boxtensor \CFDAh 
      (\sbdiff, 0)}_{n \text{ times}})
  \]
  is identically zero.
\end{proposition*}

\begin{proof}

  Let $x_1 \tensor \dotsb \tensor x_n$ be a generator of the chain complex.
  The differential applied to $x_1 \tensor \dotsb \tensor x_n$ is a sum over 
  labelled directed graphs called ``operation trees'' \cite[Definition~2.2.17, 
  par.\ before Definition~2.2.41]{LipOzsThu15}.
  The word ``tree'' is inappropriate in the Hochschild case but standard.
  The operation trees referred to here combine the definition of the 
  differential of a box tensor product $\boxtensor$ of type-\DA{} bimodules 
  \cite[Definition~2.3.9 and Figure~4]{LipOzsThu15} and the definition of
  the differential of the Hochschild operator $\CHbox$ for type-\DA{} bimodules 
  \cite[(2.3.45)]{LipOzsThu15}.

  Each of the directed graphs contributing to the differential has $n$ 
  ``input'' edges labelled by $x_1$ through $x_n$, and internal vertices 
  labelled by operations $\pi$, $\aug$, or $\DAmap^{1,\boxtensor}_{1+j}$ for 
  some $j$.  The label on each of the rest of the edges, including the $n$ 
  ``output'' edges, is obtained by applying the operation to the incoming edges 
  that are incident with its source.
  We will prove that every such operation tree contributes zero to the 
  differential on $x_1 \tensor \dotsb \tensor x_n$.
  Two facts we appeal to are that (1) if an edge is labelled with $0$, then the 
  operation tree contributes $0$; and that (2) every operation tree for the 
  Hochschild complex has exactly one vertex labelled by the augmentation 
  $\aug$.

  Let $C \subset \dga (\pmc, 0)$ be the subset consisting of all elements the 
  coefficient of whose $\rho_2$-term is $1$ in the basis in 
  \eqref{eq:dga-basis}.
  We claim that no edge can be labelled by an element of $C$.  Assume the 
  contrary, and let $e$ be the first edge from the top labelled by an element 
  of $C$.  The source $v$ of $e$ is labelled by either $\pi$ or 
  $\DAmap^{1,\boxtensor}_{1+j}$. If it is labelled by $\pi$, then $v$ is the 
  target of exactly one edge, and the label on that edge must be another 
  element of $C$, which contradicts that $e$ is the first edge from the top 
  with this property.  If $v$ is labelled by $\DAmap^{1,\boxtensor}_{1+j}$, 
  then \eqref{eq:DAbimod-box-final} shows that $v$ must also have an incoming 
  edge labeled by an element of $C$, again contradicting that $e$ is the first 
  edge with this property.

  Next,
  let $C' \subset \dga (\pmc, 0)$ be the subset consisting of all elements 
  that, in the basis in \eqref{eq:dga-basis}, have a non-zero coefficient in 
  the $\idem_{\set{0}}$-term, the $\idem_{\set{1}}$-term, or both.
  We claim that no edge can be labelled by an element of $C'$.
  Assume the contrary, and let $e$ be the first edge from the top labelled by 
  an element of $C'$.  The source $v$ of $e$ must be labelled by 
  $\DAmap^{1,\boxtensor}_{1+j}$, since $\pi$ sends all idempotents to $0$.
  The unitality of $\dga (\pmc, 0)$ and \eqref{eq:DAbimod-box-final} imply that 
  $v$ must have an incoming edge labeled by either an element of $C$ or $C'$.  
  We have ruled out a $C$-label as a possibility, and a $C'$-label would 
  contradict that $e$ is the first edge labeled with an element of $C'$.

  In particular, the unique edge incident with the unique vertex labelled by 
  $\aug$ must not be an element of $C'$.  Since the augmentation map $\aug 
  \colon \dga (\pmc, 0) \to \idemring (\pmc, 0)$ sends any non-idempotent to 
  $0$, the operation tree contributes zero to the differential on $x_1 \tensor 
  \dotsb \tensor x_n$.
\end{proof}

\begin{remark*}
  See Figure~\ref{fig:comp-heeg}. One way to summarize the argument above is 
  that, in the Heegaard diagram $\heeg$ corresponding to $\CFDAhsbdiff$, every 
  holomorphic curve whose domain does not touch the left boundary of $\heeg$, 
  or whose domain has $\strand_2$ appear on the left boundary of $\heeg$, must 
  also have $\strand_2$ appear on its right boundary.  The Hochschild complex 
  is computed from the Heegaard diagram $\heeg^{2^m}$ obtained by gluing $2^m$ 
  copies of $\heeg$ together and then gluing the leftmost boundary to the 
  rightmost boundary. The assertion about holomorphic curves in $\heeg$ then 
  ensures that there are no holomorphic curves in $\heeg^{2^m}$.
\end{remark*}

\begin{figure}[!htbp]
  \labellist
  \scriptsize\hair 2pt
  \pinlabel $\dotso$ at -15 102
  \pinlabel $\dotso$ at 654 102
  \pinlabel $z$ at 0 37
  \pinlabel $z$ at 212 37
  \pinlabel $z$ at 427 37
  \pinlabel $z$ at 638 37
  \pinlabel {$\genps$} at 105 23
  \pinlabel {$\genpp$} at 105 3
  \pinlabel {$\genrp$} at 317 23
  \pinlabel {$\genpp$} at 317 3

  \endlabellist
  \includegraphics[width=0.9\textwidth]{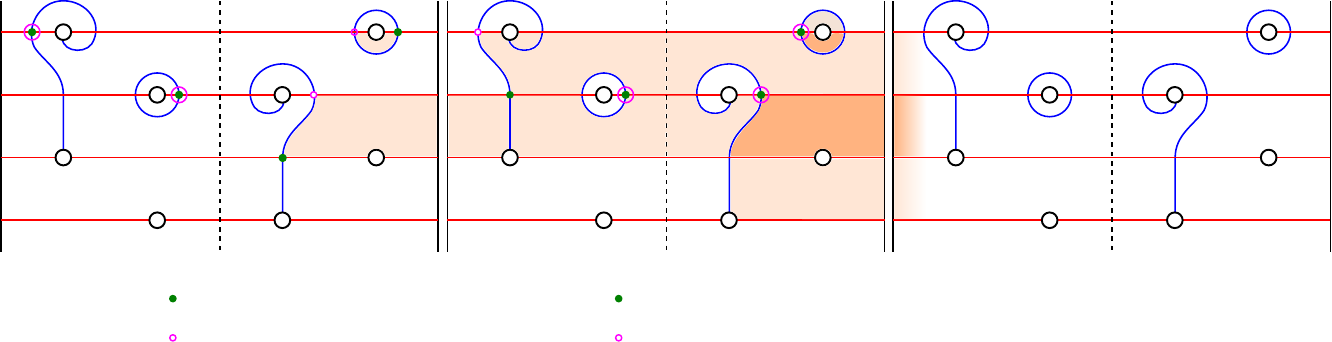}
  \caption{The non-existence of holomorphic domains in a Heegaard diagram 
    obtained by gluing together the Heegaard diagram $\heeg$ corresponding to 
    $\CFDAhsbdiff$.}
  \label{fig:comp-heeg}
\end{figure}

\subsection{Putting it together}
\label{ssec:example-conclusion}

\begin{proposition*}
  For $H = 2^m \Z_2 \supset \germ$, we have an isomorphism of vector spaces
  \[
    \HFKh (Y_H, L_H) \isom
    \HHbox (\underbrace{\CFDAh (\sbdiff) \boxtensor \dotsb \boxtensor \CFDAh 
      (\sbdiff)}_{2^m \text{ times}})
    \isom
    \bigtensor_{(2^m \Z_2 \backslash \Z_2)_S} \vsV.
  \]
  Since
  \[
    \HFKh (Y_H, L_H) \isom \SFH (M_H, L_H)
    \qquad \text{and} \qquad
    \bigtensor_{(2^m \Z_2 \backslash \Z_2)_S} \vsV \isom F_H' / F_H'',
  \]
  we have proven the theorems of \S~\ref{ssec:intro-ex} and 
  \S~\ref{ssec:example-sfh}.
\end{proposition*}

\begin{proof}

  Note that both $\boxtensor$ and $\CHbox$ respect the strands grading, and so
  \[
    \HHbox (\underbrace{\CFDAh (\sbdiff) \boxtensor \dotsb \boxtensor \CFDAh 
      (\sbdiff)}_{2^m \text{ times}})
    \isom \bigdirsum_{i \in \set{-1, 0, 1}} \HHbox (\underbrace{\CFDAh 
      (\sbdiff, i) \boxtensor \dotsb \boxtensor \CFDAh (\sbdiff, i)}_{2^m 
      \text{ times}}).
  \]

  By the Proposition of \S~\ref{ssec:example-ch}, the differential of
  \[
    \CHbox (\underbrace{\CFDAh (\sbdiff,0) \boxtensor \dotsb \boxtensor \CFDAh 
      (\sbdiff,0)}_{2^m \text{ times}})
  \]
  is identically zero; therefore, its homology is generated over $\Ftwo$ by 
  cyclic tensor products of $\genpp$, $\genps$, $\genqq$, $\genrp$, and 
  $\genrs$.  Taking into account their idempotent compatibilities, we see that 
  this homology is isomorphic to the $2^m$-fold cyclic tensor power of the 
  summand of $\bimodV$ generated by $\genu$, $\genv$, $\genw$, $\genx$, and 
  $\geny$.

  As mentioned in \S~\ref{ssec:example-bimods}, $\CFDAh (\sbdiff, -1)$ is a 
  quasi-invertible $(\dga (\pmc, -1), \dga (\pmc, -1))$-bimodule, and $\dga 
  (\pmc, -1) = \Ftwo \gen{\idem_{\emptyset}}$.  Therefore, $\CFDAh (\sbdiff, 
  -1)$ and all its tensor powers are quasi-isomorphic to the identity bimodule 
  $\Ftwo$. The Hochschild chain complex $\CHbox$ of any of these bimodules is 
  also quasi-isomorphic to $\Ftwo$ as a chain complex of vector spaces. The 
  homology of
  \[
    \CHbox (\underbrace{\CFDAh (\sbdiff,-1) \boxtensor \dotsb \boxtensor \CFDAh 
      (\sbdiff,-1)}_{2^m \text{ times}})
  \]
  is thus also isomorphic to $\Ftwo$, to $\idem_{\emptyset} \bimodV 
  \idem_{\emptyset} = \Ftwo \gen{\gent}$, and to the $2^m$-fold cyclic tensor 
  power of $\Ftwo \gen{\gent}$.

  Similarly, the homology of the summand with strands grading $1$ is isomorphic 
  to the $2^m$-fold cyclic tensor power of $\idem_{\set{0,1}} \bimodV 
  \idem_{\set{0,1}} = \Ftwo \gen{\genz}$.
\end{proof}

The generators of $\CFDAh (\sbdiff, -1)$ and $\CFDAh (\sbdiff, 1)$ that we use 
in the proof above are not discussed in \cite[\S~10.2]{LipOzsThu15}, as only 
the summand of strands grading $0$ is considered there. In each subfigure of 
\cite[Figure~25]{LipOzsThu15}, the summand of strands grading $-1$ is always 
generated by the unique generator occupying only the $\alpha$-arcs to the left 
of the $1$-handles, i.e.\ the generator occupying the two distinct points 
labeled $p$ and $q$ on the left of each subfigure.  A similar comment applies 
to $\CFDAh (\sbdiff, 1)$.

\bibliographystyle{mwamsalphack}
\bibliography{references}

\end{document}